\documentclass[a4paper,11pt]{article}
\usepackage{latexsym,amscd}
\usepackage{amsmath,amsfonts,amssymb,mathrsfs,amsthm}
\usepackage{enumerate,qsymbols,euscript}

\usepackage[all]{xy}

\newtheorem{theorem}{Theorem}[section]
\newtheorem{lemma}[theorem]{Lemma}
\newtheorem{proposition}[theorem]{Proposition}

\newtheorem{definition}[theorem]{Definition}

\newtheorem{example}{\bf Example\/}
\title{Classification of two-dimensional smooth projective algebraic semigroups}
\author{Li Duo}

\begin{document}

\maketitle{Most of this article has been written
down during a stay at Institut Fourier, supervised by Michel. Brion.\\}
\textbf{Introduction:}
\\In this article, we consider algebraic varieties
and schemes defined over the field $\mathbb C$ of complex numbers.
An algebraic semigroup is an algebraic variety endowed
with an associative composition law which is a morphism
of varieties. If in addition there is a neutral element,
then we obtain the notion of an algebraic monoid. \\The theory of linear (or equivalently, affine)
algebraic
monoids has been chiefly developed by Putcha and Renner (see the books ~\cite {Putcha} and ~\cite
{Renner}).\\Also, the structure of complete algebraic
monoids is given by a result of Mumford and Ramanujam
(\cite{Mumford} Chapter 2, Section 4, Appendix): if a complete irreducible
variety $X$ has a composition law with a neutral element,
then $X$ is an abelian variety with group law $\mu$.
\\ In this article, we address the classification of algebraic
semigroups and we consider the case of smooth projective
surfaces. The case of curves is treated in ~\cite {Michel. Brion} , which also
obtains a general structure theorem for complete algebraic
semigroups. Using that theorem, we can reduce our
classification problem to the following one: determine
all tuples.$(S, C, \pi, \sigma)$, where $\pi$ is a morphism from such a surface $S$ to a curve $C$,
with a section $\sigma$.\\ We then show that, in most cases, the algebraic semigroup structures on $S$
can be reduced to the minimal model of $S$. Then in Theorem \ref {M1}, we give a full description of
those surfaces $S$, which has at least one non-trivial algebraic semigroup structure, when their
Kodaira dimension satisfies ``$ \kappa(S)=-\infty$''. In the third part, we consider the  case ``$
\kappa (S)=0$'', and  in  Theorem \ref {M2} and  Theorem \ref {M3}, we solve this classification
problem when $S$ is bielliptic or abelian; in   Theorem \ref {M4}, we show that there are no
non-trivial algebraic semigroup structures on an  Enriques surface or a general $K3$ surface. For the
case ``$ \kappa (S)=1$'', in  Theorem \ref {M5}, we give a description of one special type of elliptic
surfaces which also admit non-trivial algebraic semigroup laws.\\
 For a given surface $S$,  it is  an interesting problem to describe all algebraic semigroup
 structures on it and determine the dimension of this moduli.  In this article, we solve this problem
 by using a result of
~\cite {Michel. Brion} , Section 4.5, Remark 16: the families of algebraic
semigroup laws on a given variety $X$ are parametrized
by a quasi-projective scheme $SL(S)$. \\Moreover, given an
algebraic semigroup structure $\mu_0$ on $S$ and the
associated contraction \xymatrix{S\ar[r]|\pi &C\ar@/_/@{>}[l]|\sigma}, the connected component
of $\mu_0$ in $SL(S)$ is isomorphic to $$Mor_{\pi}(C, S)\times A$$ where $Mor_{\pi}(C, S)$ is the
scheme of sections of $\pi$ and $A$ is an abelian variety which is  the smallest two sided ideal of
$(S, \mu_0)$. Since $Mor_{\pi}(C, S)$ can be viewed as an open subscheme of $Hilb(S)$ by assigning a
section to its image in $S$, we can use the local study of $Hilb(S)$ to determine the local dimension
of $Mor_{\pi}(C, S)$ at each section.\\By using the above ''deformation method ", we solve the problem
for those surfaces satisfying $\kappa(S)=0$ or $1$. In Theorem 5.2, we apply the function field
version of Mordell-Weil theorem to solve the problem in the case ''$\kappa(S)=2$".\\Throughout this
article, we use the books  ~\cite{Hartshone} and ~\cite{Beauville} as general references for surfaces.
\section{Definitions and Rough classification}
\begin{definition}An abstract semigroup is a set $S$ equipped with an associative composition law
$\mu :S\times S\to S $.
When $S$ is a variety and $\mu$ is a morphism, we say that $(S,\mu)$ is an algebraic semigroup.
\end{definition}
\begin{theorem}  (M.Brion ~\cite {Michel. Brion}, Section 4.3, Theorem 6.) \\Let $S$ be a complete
variety, and $\mu:S\times S\longrightarrow S$ a morphism. Then $(S,\mu)$ is an algebraic semigroup if
and only if there exist  complete varieties $X$ , $Y$,  an abelian variety $(A,+ )$ and morphisms
($\sigma$, $\pi$)
making the following diagram commutative:
\begin{displaymath}\xymatrix{X\times A \times Y \ar[r]^-\sigma\ar[d]^-{id}&S\ar[dl]^-\pi\\X\times A
\times Y}\end{displaymath}
and satisfying  $\mu(s_1,s_2)=\sigma\Big(\nu\big(\pi(s_1),\pi(s_2)\big)\Big)$, where
\[\nu\big((x_1,a_1,y_1),(x_2,a_2,y_2)\big)=\big(x_1,a_1+a_2,y_2\big).\] In particular, $\sigma$ is a
closed immersion and  a section of $\pi$.
\end {theorem}

\textbf{Notation}: We call $X\times A\times Y$ the kernel of $(S,\mu)$ and $A$ the associated abelian
variety of $(S,\mu)$.
\subsection{First steps in classification}
 Note that $\dim S=2$ and $\sigma:X\times A\times Y\hookrightarrow S$ is a closed immersion, so $\dim
 X+\dim Y+\dim A\le 2$.\\Then we list all the possibilities as follows:\\
\begin{enumerate}[1)]
\item $\dim X=\dim Y=0 , \dim A=2$. Then  $(S,\mu)=(A,+)$ is an abelian surface.
\item $\dim X=\dim Y=0 , \dim A=1 $. Then $(A,+)$ is an elliptic curve and we have the following
    commutative diagram:\\
$$\xymatrix{(A,+)\ar[r]^-\sigma\ar[d]^-{id}&(S,\mu)\ar[dl]^-\pi\\(A,+)}$$\\
where the semigroup law is given by the formula: \[\mu(s_1,s_2)=\sigma\
\big(\pi(s_1)+\pi(s_2)\big).\]

\item $\dim X=\dim Y=\dim A=0$ and $\mu$ is constant.
\item $\dim X=1$, $\dim Y=0$, $\dim A=0$ and $\mu(s_1,s_2)=\sigma(\pi(s_1))$.
\item (similar to 4)) $\dim X=0$, $\dim Y=1$, $\dim A=0$ and $\mu(s_1,s_2)=\sigma(\pi(s_2))$.
\item  $\dim X=1$, $ \dim Y=0$, $\dim A=1$. Then $S=X\times A$,  and the semigroup law is given by
    $\mu((x_1,a_1),(x_2,a_2))=(x_1,a_1+a_2)$.
\item (similar to 6))  $\dim X=0$, $\dim Y=1$, $\dim A=1$. Then $S=Y\times A$ and the semigroup law
    is given by $\mu((y_1,a_1),(y_2,a_2))=(y_2,a_1+a_2)$.
\item $\dim X=2$, $\dim Y=0$, $\dim A=0$ . Then $S=X$, and the semigroup law is given by
    $\mu(s_1,s_2)=s_1$.
\item (similar to 8))    $\dim X=0$, $\dim Y=2$, $\dim A=0$. Then $S=X$, and the semigroup law is
    given by $\mu(s_1,s_2)=s_2$.
\item $\dim X=\dim Y=1$. Then $S= X \times Y$  and the semigroup law is given by
    $\mu((x_1,y_1),(x_2,y_2))=(x_1,y_2)$.
\end{enumerate}
\textbf{Remark 1):} \\We call cases 1) 3) 6) 7) 8) 9) 10) \textbf {trivial} and they will not be
considered any more.\\
In cases 2) 4) 5), we observe that $S$ always maps to a curve $C$ via  some morphism
$\pi:S\longrightarrow C $ with a section $\sigma: C\longrightarrow S$, and there is always an
algebraic semigroup law $\tilde\mu$ on $C$, such that the semigroup law $\mu$ on $S$ is given
by:\begin{equation} \label{law}
\mu(s_1,s_2)=\sigma\big(\tilde\mu(\pi(s_1),\pi(s_2))\big).\end{equation} To be concrete, in case 2),
$(C,\tilde\mu)=(A,+)$; in case 4), $C=X$, and for all $ x_1,x_2\in X$, $ \tilde\mu(x_1,x_2)=x_1$; in
case 5), $C=Y$, and for all $ y_1,y_2\in Y$, $ \tilde\mu(y_1,y_2)=y_2$.  We call the semigroup laws on
$S$ in these cases \textbf {nontrivial}. In the following, we deal with  non-trivial algebraic
semigroup structures, and we denote the associated contraction morphism $\pi$ and its section
$\sigma$, by the following diagram \xymatrix{S\ar[r]|\pi &C\ar@/_/@{>}[l]|\sigma}.\begin{lemma}If
\xymatrix{S\ar[r]|\pi &C\ar@/_/@{>}[l]|\sigma}, then $\pi_*\mathcal O_S=\mathcal O_C$.\end{lemma}
\begin{proof} ~\cite {Michel. Brion}, Section 4.3, Lemma 1.\end{proof}
From this lemma, we deduce that $\pi$ is a fibration, (which means every fibre of $\pi$ is connected),
and the projective curve $C$ is normal, hence smooth.
Now in order to solve the problem, our task becomes the following one: \\classify all
\xymatrix{S\ar[r]|\pi &C\ar@/_/@{>}[l]|\sigma}, where $\pi$ is a fibration and $\sigma$ is a section
of $\pi$.
\subsection{Reduction of  the problem to the minimal model}
In this subsection, we keep the assumption that \xymatrix{S\ar[r]|\pi &C\ar@/_/@{>}[l]|\sigma}. When
$g(C)\ge1$, we  prove that $S$ has a non-trivial algebraic semigroup structure if and only if its
minimal model has one. This will  reduce our problem to the case that $S$ is minimal. In Proposition
1.4, we consider the   blowing-up of a point on $S$, and we show that any non-trivial algebraic
semigroup structure can be lifted to this blowing-up. Then we give an example to show that in some
cases, after blowing-down an exceptional rational curve on $S$, the non-trivial semigroup structures
are not  preserved. In Proposition 1.5, we show that, if $g(C)\ge1$, any non-trivial algebraic
semigroup structure is preserved after contracting an exceptional rational curve.\\
The main results of this subsection are Proposition \ref {P1} and Proposition \ref {P2}.
\begin{proposition}\label{P1}
Assume that the non-trivial semigroup law $\mu$ on $S$ is given by the formula (1) in Remark 1). Let
$\xymatrix{\varphi: S'\ar[r] &S}$ denote the blowing-up of  an arbitrary point $P$ on $S$. Then $S'$
has a unique non-trivial algebraic semigroup structure such that  $\varphi$ is a homomorphism.
\end{proposition}Now let us begin the proof of Proposition 1.4.
\begin{proof} First, we  construct a morphism from $S'$ to $C$, and a section of this morphism; then
we construct a non-trivial algebraic semigroup structure on $S'$, and verify that $\varphi$ is a
homomorphism. \\ Composing $\pi$ with $\varphi$, we get a morphism from $S'$ to $C$, denoted by
$\pi'=\pi\circ\varphi:S'\longrightarrow C$.\\ We now construct a section of $\pi'$. Since $\sigma$
maps $C$ isomorphically to its image $\sigma(C)$, in order to construct a morphism from $C$ to $S'$,
we consider the strict transform of $\sigma (C)$, and denote it by $C'$. \\Then
$\varphi_{|C'}:C'\longrightarrow \sigma(C)$  is the blowing-up of the point $P $ on $\sigma( C)$.
Since $\sigma(C)$ is smooth,  $\varphi_{|C'}$ is an isomorphism. Let
$\sigma'=\varphi_{|C'}^{-1}\circ\sigma$, then  $\sigma'$ is a morphism from $C$ to $S'$. Next let us
verify that $\sigma'$ is a section of $\pi'$ :\begin{equation}\label{2}
\pi'\circ\sigma'=\pi\circ\varphi\circ\sigma'=\pi\circ\sigma= id_C.\end{equation} Thus we construct a
semigroup law $\mu'$ on $S'$ by using the formula (1) in Remark 1). We now verify that $\varphi'$ is a
homomorphism of semigroups. Let  $s'_1,s'_2 \in S'$, then \begin{equation}
\xymatrix{\mu(\varphi(s'_1),\varphi(s'_2))
=\sigma(\tilde\mu(\pi\circ\varphi(s'_1),\pi\circ\varphi(s'_2)))
=\sigma(\tilde\mu(\pi'(s'_1),\pi'(s'_2)))}\end{equation}
and we also have the following equations:\begin{equation}
\xymatrix{\varphi(\mu'(s'_1,s'_2))=\varphi(\sigma'(\tilde\mu(\pi'(s'_1),\pi'(s'_2)))=\sigma(\tilde\mu(\pi'(s'_1),\pi(s'_2)))}.\end{equation}
So  we get $\mu(\varphi(s'_1),\varphi(s_2'))=\varphi(\mu'(s'_1,s'_2))$, which means that $\varphi$ is
a homomorphism of  semigroups.\\ Finally the uniqueness of the semigroup law making $\varphi$ a
homomorphism is due to the fact that  $\xymatrix{S'\ar[r]^\varphi &S}$ is birational. \end{proof}
We already know that blowing-up of  a point on $S$ preserves a non-trivial algebraic semigroup
structure, what happens if we perform a blowing-down? The next example shows that the non-trivial
semigroup structures are not necessarily preserved.
 \begin{example} Consider the blowing-down morphism:$$ \varphi:Proj_{\mathbb P^1}(\mathcal
 O\oplus\mathcal O(-1))\longrightarrow  \mathbb P^2. $$ We show that there exist non-trivial algebraic
 semigroup structures on \\$Proj_{\mathbb P^1}(\mathcal O\oplus\mathcal O(-1))$, but every algebraic
 semigroup structure on $ \mathbb P^2 $ is trivial. \\ Note that $Proj_{\mathbb P^1}(\mathcal
 O\oplus\mathcal O(-1))$ is ruled over $\mathbb P^1$ and the ruling morphism $\pi$ has sections. For
 each section, we can define a semigroup law on $Proj_{\mathbb P^1}(\mathcal O\oplus\mathcal O(-1))$
 as in Remark 1). For example, recall the projection morphism\\ \xymatrix{\mathcal O\oplus\mathcal
 O(-1)\ar[r]& \mathcal O\ar[r]&0} on $\mathbb P^1$ corresponds to a section:\\ \xymatrix{\sigma:
 \mathbb P^1\ar[r]& Proj_{\mathbb P^1}(\mathcal O\oplus\mathcal O(-1))}. We can thus define a
 non-trivial semigroup law $\mu$ on  $Proj_{\mathbb P^1}(\mathcal O\oplus\mathcal O(-1))$ by
 $\mu(x_1,x_2)=\sigma\circ \pi (x_1)$.
\\ Also recall that for an arbitrary smooth projective curve $C$, there are only constant morphisms
from $\mathbb P^2$ to $C$, which means that there will be no retraction from $\mathbb P^2$ to a curve.
It follows that the only possible algebraic semigroup structures  on $\mathbb P^2 $ are trivial.
 \end{example}\begin{proposition}\label{P2}Let  $(S,\mu)$ satisfy the same assumptions as in
 Proposition \ref{P1}. Let $\varphi: S\longrightarrow \tilde S$ denote the blowing-down of an
 exceptional rational curve $E$ on $S$. Furthermore, we require that $\pi(E)$ is a point $P$. Then
 there exists a unique non-trivial semigroup law $(\tilde S,\tilde \mu)$ such that $\varphi$ is a
 homomorphism.  \end{proposition}
Now let us begin the proof of Proposition 1.5.
\begin{proof}First we construct a morphism from $C$ to $\tilde S$, then we construct a contraction of
this morphism. After defining a non-trivial semigroup structure on $\tilde S$, we verify that
$\varphi$ is a homomorphism.\\
Composing  $\varphi$ with $\sigma$, we get a morphism $\tilde
\sigma=\varphi\circ\sigma:C\longrightarrow \tilde S$.\\
 Since $\varphi: S-E \longrightarrow \tilde S-P $ is an isomorphism,  there is a set-theoretical map
 $\tilde \pi$ from $\tilde S$ to $C$ satisfying $\pi=\tilde \pi\circ\varphi$.\\
Let us verify that $\tilde \pi$ is a morphism. It is easy to see $\tilde \pi$ is continuous, because
the preimage of any finite set is closed in $\tilde S$. We now define $\tilde \pi ^ {\sharp}: \mathcal
O_C\longrightarrow \tilde \pi_*\mathcal O_{\tilde S}$. For any open subset $U$ of $C$, there is a ring
homomorphism $\pi_{U}^{\sharp}:\mathcal O_C(U)\longrightarrow \mathcal O_S(\pi^{-1}(U))$. Since
$\varphi_*\mathcal O_S=\mathcal O_{\tilde S}$,  there is an isomorphism $\varphi_{\tilde
\pi^{-1}(U)}^{\sharp}:\mathcal O_{\tilde S}(\tilde \pi^{-1}(U))\longrightarrow \mathcal
O_S(\pi^{-1}(U))$. We define $\tilde \pi ^ {\sharp}_U: \mathcal O_C (U)\longrightarrow \mathcal
O_{\tilde S}(\tilde \pi^{-1}(U))$ by $\tilde \pi ^ {\sharp}_U=\varphi_{\tilde
\pi^{-1}(U)}^{\sharp-1}\circ \pi_{U}^{\sharp}$. Then $(\tilde \pi,\tilde \pi^{\sharp})$ is a morphism
of ringed spaces, so we get a morphism $\tilde \pi$ from $\tilde S$ to $C$.\\
 Now we verify $\tilde \sigma$ is a section of $\tilde \pi$: \begin{equation} \tilde \pi\circ\tilde
 \sigma=\tilde \pi\circ\varphi\circ\sigma=\pi\circ\sigma=id_C.\end{equation}\\ Thus we have
 constructed a non-trivial  algebraic semigroup structure $\tilde \mu$ on $\tilde S$ as in Remark
 1).\\ Let us verify that $\varphi$ is a homomorphism of semigroups: for any $s_1,s_2\in S$,
 \begin{equation}
\varphi(\mu(s_1,s_2))=\varphi\circ\sigma(\tilde\mu(\pi(s_1),\pi(s_2))) \end{equation}and
\begin{equation}\tilde \mu(\varphi(s_1),\varphi(s_2))=\tilde \sigma\circ\tilde\mu(\tilde
\pi\circ\varphi(s_1),\tilde \pi\circ\varphi(s_2)). \end{equation}
Since \begin{equation} \tilde \pi\circ\varphi=\pi  \end{equation}and \begin{equation}\tilde
\sigma=\varphi\circ\sigma \end{equation}we have $\mu(\varphi(s_1),\varphi(s_2))=\varphi(\tilde
\mu(s_1,s_2))$, which means that $\varphi$ is a homomorphism.\\Finally, $S$ and $\tilde S$ are
birationally equivalent, so there is a unique non-trivial algebraic semigroup structure on $\tilde S$
such that $\varphi$ is a homomorphism.\end{proof}
In view of the last two propositions, when $g(C)\ge 1$, we can can reduce our problem to the case
where $S$ is minimal, the case $g(C)=0$ need further study. But whatever in what follows, we always
assume $S$ is minimal.
\section{The case $\kappa(S)=-\infty$}In this section, we always assume  $\kappa(S)=-\infty$, and our
aim  is to prove:

\begin{theorem}[Main theorem]\label {M1}
$$\textrm{If}\xymatrix{S\ar[r]|\pi &C\ar@/_/@{>}[l]|\sigma}. $$\begin{enumerate}[1)]
\item If $g(C)\ge 1$,  then $\pi$ is a ruling morphism.
\item If $g(C)=0$ and the general fibre of $\pi$ is rational,\\then $S\simeq Proj_{\mathbb
    P^1}(\mathcal O\oplus\mathcal O(-d))  $ for some $d\ne 1$ and $\pi$ is the ruling morphism.
\item If $g(C)=0$, i.e. $C\simeq \mathbb P^1$ and the general fibre of $\pi$ is not rational,\\then
    $S\simeq \mathbb P^1\times X$, where X is a projective smooth curve satisfying $g(X)\ge1$ and
    $\pi=pr_{1}$ is the first projection to $\mathbb P^1$.\end{enumerate}\end{theorem}
The parts 1) and 2) of  Theorem \ref {M1} are more or less trivial. In order to prove 3), we  analyze
the cone of curves of $S$, $NE(S)$. To be more precise, we show that $NE(S)$ is two-dimensional, and
give an explicit description of the extremal rays of this convex cone.
Before proving Theorem \ref{M1}, we need some general facts about ruled surfaces.
 \begin{definition}{(definition and notation)}\begin {enumerate}[a)]\item (definition of normalised
 sheaf).\\If $\pi:S\longrightarrow C$ is a ruled surface, then there exists a rank two locally free
 sheaf $\mathcal E$ on $C$ s.t. $S\simeq \mathbb P_C(\mathcal E)$, $H^0(\mathcal E)\ne 0$ and for an
 arbitrary invertible sheaf $\mathcal L$ with $deg(\mathcal L)<0$ on $C$, $H^0(\mathcal E\otimes
 \mathcal L)=0$ (the existence of such $\mathcal E$ can be found in ~\cite {Hartshone}, Chapter 5,
 Proposition 2.8).We call such $\mathcal E$ normalised. \item (definition of invariant $e$)\\For any
 normalised sheaf $\mathcal E$, we define $e$ by \begin{equation} e=-deg(\mathcal E)=-deg(\bigwedge^2
 \mathcal E). \end{equation}This number is a well-defined invariant of $S$, not depending on the
 $\mathcal E$ we choose.\\(The proof that $e$ is a well-defined invariant, can be found in ~\cite
 {Hartshone}, Chapter 5, Proposition 2.8).\item (some notation)\\In $Pic(S)$, we let $C_0$ denote the
 linear equivalence class of $\mathcal O_{\mathbb P_C(\mathcal E)}(1)$. And we let $f$ denote the
 numerical equivalence class of any general fibre of $\pi$. \end{enumerate}\end{definition}
 We quote the following lemmas without proof, they can be found in the book ~\cite{Hartshone},
 Proposition $2.6$, Proposition $2.20$ and Proposition $2.21$. We will use these lemmas to describe
 $NE(S)$.
\begin {lemma}Let $S\simeq \mathbb P_C(\mathcal E)$ be a ruled surface over $C$, for some  normalised
sheaf $\mathcal E$. Then there is a one-to-one correspondence between sections of $\pi$ and
surjections $\mathcal E\longrightarrow \mathcal L\longrightarrow 0$, where $\mathcal L$ is an
invertible sheaf on $C$.
There exists a  section $\sigma$ such that  $\sigma(C)=\mathcal O_{\mathbb P_C(\mathcal E)}(1)$ in
$Pic(S)$ ($\sigma$ may not be unique), $\sigma$ is called a canonical section of $\pi$.\end{lemma}
 \begin{lemma}If $\pi:S\longrightarrow C$ is a ruled surface, then \\a) $Pic(S)=\mathbb Z C_0\oplus
 \pi ^* Pic(C) $. \\ b) The self-intersection number satisfies $C_0^2=-e$.\\
 c) The canonical line bundle of $S$ satisfies $K_S\sim_{num}-2C_0+(2g(C)-2-e)f$.\end{lemma}
 \begin{lemma}\label{5}If $\pi:S\longrightarrow C$ is a ruled surface, with $e\ge 0$
 then:\begin{enumerate}[a)]\item For any irreducible curve $Y$ on $S$, $Y$ is numerically equivalent
 to $a C_0+b f$, for some $a,b\in \mathbb Z$. If $Y$ is not numerically equivalent to $C_0$ or $ f$,
 then  $a>0$ , $b\ge ae$.
 \item A divisor $D\sim_{num}aC_0+bf$ is ample if and only if $a>0$ and $b>ae$.\end{enumerate}
\end{lemma}
\begin{lemma}\label{6}If $\pi:S\longrightarrow C$ is a ruled surface, with $e\less 0$ then
:\begin{enumerate}[a)]\item For any irreducible curve $Y$ on $S$, $Y$ is numerically equivalent to $a
C_0+b f$ for some $a,b\in \mathbb Z$. If $Y$ is not numerically equivalent to $C_0$ or $ f$, then
either $a=1$, $b>0$ or $a\ge2$, $2b\ge a e$.\item A divisor $D\sim_{num}a C_0+ b f$ is ample if and
only if $a>0$, $2b>a e$. \end{enumerate}\end {lemma}
The next lemma is a version of the ``Rigidity lemma''. (See Lemma 1.15, ~\cite {Olivier Debarre} )For
convenience, we give a detailed proof in the following.
\begin{lemma}{(Rigidity lemma)}\\ Assume that there are two fibrations $\pi_1:S\longrightarrow C_1$,
$\pi_2:S\longrightarrow C_2$ where $C_1$, $C_2$ are smooth curves.  If $\pi_2$ contracts one fibre
$F_1$ of $\pi_1$, then it contracts all fibres of $\pi_1$  and there exists an isomorphism
$\theta:C_1\longrightarrow C_2$, s.t. $\pi_2=\theta\circ\pi_1$, which means $\pi_1$ and $\pi_2$ are
isomorphic as fibrations.  \end{lemma}
\begin{proof}First we prove that $\pi_2$ contracts all fibres of $\pi_1$. Observe that $\pi_2$
contracts $F_1$ if and only if $\pi_{2*}(F_1)\sim_{num}0$. Since the fibres of $\pi_1$ are
parameterized by $C_1$, they are all numerically equivalent. So for an arbitrary fibre $F$ of $\pi_1$,
we have $\pi_{2*}(F)\sim_{num}0$,  which implies that $\pi_2$ also contracts $F$.\\
 Then $\pi_2$ factors through $\pi_1$ set-theoretically, i.e. there exists a map $  \theta $ s.t.
 $\pi_2=\theta\circ\pi_1$.\\ We now show that $\theta$ is a morphism. \\Pick an arbitrary open set $U$
 of $C_2$, we have $\pi_2^{-1}(U)=\pi_1^{-1}\circ\theta^{-1}(U)$. Since $\pi_1$ is surjective,
 $\pi_1(\pi_2^{-1}(U))=\theta^{-1}(U)$. Because $\pi_1$ is flat,  it is an open map. Since $\pi_1$
 maps the open set $\pi_2^{-1}(U)$ onto  $\theta^{-1}(U)$, so $\theta^{-1}(U)$ is open in $C_1$ . Now
 we have proved that $\theta$ is continuous.\\ Then we construct $\theta^ { \sharp}:\mathcal
 O_{C_2}\longrightarrow \theta_*\mathcal O_{C_1}$. Since both $\pi_i$ have connected fibres,
 $\pi_{i*}\mathcal O_S=\mathcal O_{C_i}$. So for an arbitrary open subset $U$ of $C_2$, both
 $\pi_{1}^{\sharp}:\mathcal O_{C_2}(U)\longrightarrow \mathcal O_S(\pi_2^{-1}(U))$ and
 $\pi_{2}^{\sharp}:\mathcal O_{C_1}(\theta^{-1}(U))\longrightarrow \mathcal O_S (\pi_2^{-1}(U))$ are
 isomorphisms. We define $\theta_{U}^{\sharp}:\mathcal O_{C_2}(U)\longrightarrow
 O_{C_1}(\theta^{-1}(U))$ by $\theta_{U}^{\sharp}=\pi_{2}^{\sharp-1}\circ\pi_{1}^{\sharp}$, then
 $(\theta, \theta^ { \sharp})$ is a morphism of ringed spaces. Now we have proved that $\theta$ is a
 morphism of varieties.\\ Then $\theta$ is a finite morphism between smooth projective curves. Since
 both $\pi_i$ have connected fibres, $deg(\theta)=1$, which implies $\theta$ is an isomorphism.
\end{proof} Recall the following result from the book  ~\cite{BPV}, Theorem 18.4, Chapter 3:
\begin{theorem}{(Itaka's conjecture $C_{2,1}$).} Let $\varphi:S\longrightarrow C$ be a fibration, then
the following  inequality holds for any general fibre $S_c$ :
\begin{equation} \kappa(S_c)+\kappa(C)\le \kappa(S). \end{equation}\end{theorem}
Now let us begin the proof of Theorem \ref{M1}.
\begin{proof}
\begin{enumerate}[1)]\item Assume that $g(C)\ge 1$. Then for any general fibre $F_0 $ of $\pi$, by
Theorem 2.1,  $\kappa(F_0)+\kappa(C)\le \kappa(S)=-\infty$. Then $\kappa(F_0)=-\infty $, which means
$F_0$ is smooth rational, so $\pi$ is a ruling morphism.
 \item If $g(C)=0$ and general fibres of $\pi$ are smooth rational curves, then $S$ is ruled over
     $C\simeq \mathbb P^1$, which implies that $S$ is rational. Since $S$ is minimal, $S$ is
     $\mathbb P^2$ or  $S\simeq Proj_{\mathbb P^1}(\mathcal O\oplus\mathcal O(d))  $ for some $d\ne
     1$. But every morphism from $\mathbb P^2$ to $\mathbb P^1$ is constant, so there will be no
     sections of $\pi$. So $S\simeq \mathbb P^2$ is impossible, and we get $S\simeq Proj_{\mathbb
     P^1}(\mathcal O\oplus\mathcal O(d))  $ for some $d\ne 1$.
 \item
In this paragraph, we show that $S$ is ruled over some smooth curve $X$ and introduce some notation.
Since $\kappa(S)=-\infty$ and $S$ is minimal, $S$ is a ruled surface over some projective smooth
curve $X$. We denote the ruling morphism by $\varphi:S\longrightarrow X$, and write $S$ as $\mathbb
P_X(\mathcal E)$ for some normalised sheaf $\mathcal E$ on $X$.  Using the notation of Definition
2.3, we denote the canonical section of $\varphi$ by $\tau$. Then we have the  following diagram
involving two morphisms with sections:\\$$\xymatrix{S\ar[r]|\varphi
\ar[d]|\pi&X\ar@/_/@{>}[l]|\tau\\\mathbb P^1\ar@/_/@{>}[u]|\sigma}$$
    \\We now show that $NE(S)$ is closed. Let $F_0$ be a general fibre of $\pi$ and denote its
    numerical equivalence  class  by $f_0$. We show that $F_0$ can't be contracted by $\varphi$.
    Otherwise, by the ``Rigidity lemma'', $\varphi$ will factor through $\pi$. Then $\varphi$ ,
    $\pi$ will be isomorphic as fibrations, the general fibres of $\pi$ will be rational, which
    contradicts our assumptions of 3). So $f_0$ is not a multiple of $f$, which implies that the
    extremal lines generated by them are different. Since the second Betti-number, $b_2(S)=2$,  the
    cone $NE(S)$ is two-dimensional, and its boundary $\partial NE(S)$ consists of two  extremal
    lines. Since the self-intersection numbers $f_0^2$ and $f^2$ are zero, by the ``cone theorem for
    surfaces''(Lemma 6.2, Chapter 6, ~\cite {Olivier Debarre} ), we know that both $f_0$ and  $f$
    lie in $\partial NE(S)$. It follows that $\partial NE(S)$ consists exactly of the two extremal
    lines  generated by $f_0$ and  $f$.  Since $f_0$ and  $f$ belong to $NE(S)$, $\partial
    NE(S)\subset NE(S) $, which means $NE(S)$ is closed.\\In the following, we aim to prove that the
    intersection number $f_0\cdot f=1$. \\The two cases $e< 0$ , $e\ge0$ will be dealt with
    separately.\\
    a)
    The  case $e<0$ :\\First, we determine the ample cone and the nef cone  of $S$.\\By Lemma
    \ref{6}: we have  \\$Amp(S)$= \{$aC_0+b f|a>0,\quad 2b>a e$\}.\\
    Taking the closure of $Amp(S)$, we get the nef cone:\\ $Nef(S)=\{aC_0+bf|a\ge0,\quad 2b\ge
    ae\}$.\\ \\Then, we determine the numerical equivalence class of $f_0$.\\Let $f_0=a_1C_0+b_1f$,
    where $a_1=1,\quad b_1\ge 0$ or $a_1\ge 2,\quad 2b_1\ge a_1 e$. \\
    If $a_1=1$ and $f_0=C_0+b_1f$,  $b_1\ge 0$, then $f_0\in Amp(S)$. This contradicts  $f_0^2=0$.\\
    So $f_0=a_1C_0+b_1f$, where $a_1\ge2$, $ 2b_1\ge ae $.\\ Let us determine the numbers $a_1$ and
    $b_1$ by calculating the intersection numbers $f_0^2$ and $f_0\cdot\sigma(\mathbb P^1)$.\\ Since
    \begin{equation} 0=f_0^2=a_1^2C_0^2+2a_1b_1=2a_1b_1-ea_1^2\end{equation} we get $2b_1=a_1e$,
    which implies $f_0= \frac{a_1}{2}(2C_0+ef)$.\\Let $d=(2C_0+ef)\cdot\sigma(\mathbb P^1)$, since
    $f_0\cdot\sigma(\mathbb P^1)=1$, $a_1d=2$. Note that $a_1\ge2$, so $a_1=2$, $b_1=e$. Hence
    \begin{equation} f_0=2C_0+ef.\end{equation}\\Then we  show that $f_0\cdot f=1$.\\ If
    $\sigma(\mathbb P^1)= f$, then $f_0\cdot f=f_0\cdot\sigma(\mathbb P^1)=1$, we are done.\\Our aim
    is to exclude the case $\sigma(\mathbb P^1)\ne f$. Since $f_0\cdot\sigma(\mathbb P^1)=1$,
    $\sigma(\mathbb P^1)$  is not a multiple of $f_0$. If $\sigma(\mathbb P^1)\ne f$, by Lemma
    \ref{6},  $\sigma(\mathbb P^1)=a_2C_0+b_2f$, where $a_2=1,\quad b_2\ge 0$ or $a_2\ge 2,\quad
    2b_2> a_2 e$, in both cases, $(a_2,b_2)$ satisfies  $a_2>0,\quad 2b_2>a_2 e$, which means
    $\sigma(\mathbb P^1)$ lies in  $Amp(S)$. \\
    Consider the canonical divisor $K_S$: \begin {align}
    K_S&=-2C_0+(2g(X)-2-e)f\\&=(-2C_0-ef)+(2g(X)-2)f\\ &=-f_0+(2g(X)-2)f.\end{align}
    \\So by Riemann-Roch theorem:\begin {align} &g(\sigma(\mathbb P^1))\\&=1+
    \frac{1}{2}(\sigma(\mathbb P^1)^2+\sigma(\mathbb P^1)\cdot K_S)\\&=1+\frac{1}{2}\sigma(\mathbb
    P^1)^2+\frac{1}{2}\{-f_0+(2g(X)-2)f\}\cdot\sigma(\mathbb P^1).\end{align}\\ Recall a result of
    Nagata ~\cite{Nagata}, Theorem 1:\\``Let $ S$ be a $\mathbb P^1 $-bundle over a smooth curve $X$
    of genus $g(X)$, then the invariant $e\ge-g(X)$.''\\So $g(X)\ge -e>0$. Since $\sigma(\mathbb
    P^1)$ is ample, $\sigma(\mathbb P^1)^2>0$ and $\sigma(\mathbb P^1)\cdot f >0$. We have
    \begin{equation}g(\sigma(\mathbb P^1))>1-\frac{1}{2}f_0\cdot\sigma(\mathbb
    P^1)=1-\frac{1}{2}>0\end{equation} which is impossible.\\

  So finally we get $\sigma(\mathbb P^1)=f$ and $f\cdot f_0=\sigma(\mathbb P^1)\cdot f_0=1$\\
    b) The case $e\ge 0$ \\We now show $f_0=C_0$.\\
    By Lemma \ref{5} a), we have $f_0=a_1C_0+b_1f$  where $a_1\ge 0$ and $b_1\ge a_1e$ or
    $a_1=1,b_1=0$.\\
    If $b_1>a_1e$. Then by Lemma 2.6 b), $f_0$ is ample, which is impossible.\\
    So $b_1=a_1e$, $f_0=a_1(C_0+ef)$. Since $f_0^2=0$,  $a_1^2(C_0^2+2e)=a_1^2 e=0$. So $b_1=a_1e$
    must be zero, and $f_0=a_1C_0$. But $\sigma(\mathbb P^1)\cdot f_0=1$, so  $a_1 \sigma(\mathbb
    P^1)\cdot C_0=1$, which implies that $a_1$ must be 1. So $f_0=C_0$ and $f_0\cdot f=1$. \\ In
    conclusion of  our analysis  of cases a) and  b), we get that $f_0\cdot f=1$.\\
    Define $\alpha :S \longrightarrow \mathbb P^1 \times X$ by $\alpha=(\pi,\varphi)$, we show
    $\alpha$ is an isomorphism. \\
    Assume there is an irreducible curve $C'$ on $S$ contracted by $\alpha$, then  $C'$ will be
    contracted by both $\pi$ and $\varphi$. Since every fibre of $\varphi$ is integral, $C'$ will be
    some fibre $F'$ of $\varphi$. But we already know that no fibre of $\varphi$ can be contracted
    by $\pi$. So $\alpha$ contracts no irreducible curves on $S$. This implies that  $\alpha$ is a
    quasi-finite morphism. Since $\alpha$ is also projective, it is a finite morphism.\\Since
  $deg(\alpha)=f_0\cdot f=1$, $\alpha$ is an isomorphism. So  we get $S\simeq \mathbb P^1\times X$
  and $\pi=pr_1.$\end {enumerate}
    \end{proof}For a given surface $S$, it is natural to ask how many  algebraic semi-group laws
    exist on it ? In the next theorem we solve this problem for minimal rational surfaces.
    \begin{theorem}If $S$ is a minimal rational surface, then there are finitely many algebraic
    semigroup laws on $S$ modulo $Aut(S)$. \end{theorem}
    \begin{proof}Let ($S$, $\mu$)be a non-trivial algebraic semigroup law, we denote its associated
    contraction by the following diagram: \xymatrix{S\ar[r]|\pi & C\ar@/_/@{>}[l]|\sigma}, where $C$
    is the kernel of $\mu$. Since ($S$, $\mu$) is non-trivial, $C$ is a curve. Since $S$ is
    rational, $C\simeq  \mathbb P^1 $ and $\pi$ is a ruling morphism. So $S\simeq Proj_{\mathbb
    P^1}(\mathcal O\oplus\mathcal O(-d))$ for some $d\ne 1$ . Then sections of the ruling morphism
    $\pi$ are in one-to-one correspondence with surjections \xymatrix{\mathcal O\oplus\mathcal
    O(-d)\ar[r]^{\quad p}& \mathcal L\ar[r]&0}, where $\mathcal L$ is an invertible sheaf on
    $\mathbb P^1$. Consider the kernel of $p$, we denote it by $ \mathcal N$, then $\mathcal N$ is
    also an invertible sheaf. Now there is an exact sequence  \xymatrix{0\ar[r]&\mathcal
    N\ar[r]&\mathcal O\oplus\mathcal O(-d)\ar[r]^{\quad p}& \mathcal L\ar[r]&0}. Observe that
    $deg(\mathcal N)\le 0$, we let $\mathcal N=\mathcal O(-d_2)$ and $\mathcal L=\mathcal O(d_1)$,
    where $d_2\ge 0$ and both $d_i$ are integers.\\  Case 1), $d_2>0$. \\We consider the long-exact
    sequence:\\
  $$\xymatrix{0\ar[r]&H^0(\mathbb P^1, \mathcal O (-d_2))\ar[r]&H^0(\mathbb P^1, \mathcal O\oplus
  \mathcal O (-d_2))\ar[r]^{\quad p_1}&H^0(\mathbb P^1, \mathcal L)}.$$ Since $\dim H^0(\mathbb P^1,
  \mathcal O (-d_2))=0$, $p_1$ is injective. So $\dim H^0(\mathbb P^1, \mathcal L)\ge 1$, which
  implies that  $deg(\mathcal L)\ge 0$. Consider \begin{align*}
Ext^1_{\mathbb P^1}(\mathcal N, \mathcal L) & = Ext^1_{\mathbb P^1}(\mathcal O(-d_2), \mathcal
O(d_1)) \\
& = Ext^1_{\mathbb P^1}(\mathcal O , \mathcal O(d_1+d_2))\\&=H^1(\mathbb P^1, \mathcal
O(d_1+d_2))\\&=H^0(\mathbb P^1, \mathcal O(-2-d_1-d_2)).
\end{align*}Since $d_1=deg(\mathcal L)\ge 0$, $d_2>0$, $\dim H^0(\mathbb P^1, \mathcal
O(-2-d_1-d_2)=0$. So any extension of $\mathcal N$ by $\mathcal L$ is trivial, which implies
$\mathcal N \oplus \mathcal L\simeq \mathcal O \oplus \mathcal O(-d) $. Since $\mathcal O \oplus
\mathcal O(-d) $ is normalised, $deg(\mathcal L)\le 0$, but we already know that $deg(\mathcal L)\ge
0$, so $\mathcal L \simeq \mathcal O$. Observe that \begin{align*}\Lambda ^2(\mathcal
N\oplus\mathcal L)= \mathcal N\otimes \mathcal L=\mathcal O(-d).\end{align*}So $\mathcal N\simeq
\mathcal O(-d)$ and there exists a commutative  diagram:\\
$$\xymatrix{0\ar[r]&\mathcal O(-d)\ar[r]\ar[d]_{\simeq}&\mathcal O\oplus\mathcal
O(-d)\ar[r]\ar[d]_{\theta}&\mathcal O\ar[r]\ar[d]_{\simeq}&0\\0\ar[r]&\mathcal N\ar[r]&\mathcal
O\oplus\mathcal O(-d)\ar[r]&\mathcal L\ar[r]&0} $$where $\theta$ is an automorphism. \\Case 2),
$d_2=0$, .\\We have $\mathcal N\simeq \mathcal O$ and $\mathcal L =\mathcal O(-d)$. Consider the
surjection  $$\xymatrix{\mathcal O\oplus\mathcal O(-d)\ar[r]^{\quad p}& \mathcal L\ar[r]&0},$$ then
$p\in Hom(\mathcal O \oplus \mathcal O(-d), \mathcal O(-d))=Hom(\mathcal O(-d), \mathcal
O(-d))=\mathbb C$, so $p$ is a multiplication by some non-zero scalar $\lambda\in \mathbb C$.
Consider  the isomorphism $\alpha=(id, \times \lambda): \mathcal O \oplus \mathcal
O(-d)\longrightarrow \mathcal O \oplus \mathcal O(-d)$, then  it fits into the following commutative
diagram:\\
 $$\xymatrix{\mathcal O \oplus \mathcal O(-d)\ar[r]^p \ar[d] _{\alpha}&\mathcal
 L\ar[d]_{\simeq}\ar[r]&0\\\mathcal O \oplus \mathcal O(-d)\ar[r]^{p_2}&\mathcal O(-d)\ar[r]&0 },
 $$where $p_2$ is the second projection. \\ In conclusion, modulo $Aut(\mathcal O \oplus \mathcal
 O(-d))$, there are finitely many surjections $\mathcal O \oplus \mathcal O(-d)\longrightarrow
 \mathcal L$. So modulo $Aut(S)$, there are finitely many sections of $\pi$.
\end{proof}
    \section{Classification in the case $\kappa(S)=0$} In this section, we assume that the Kodaira
    dimension of $S $ equals zero.\\ First we state a classification theorem, which is part of
    ``Enriques Classification Theorem''. We  recall some useful notations: $p_g=\dim H^0(S, \mathcal
    O_S(K_S))$, $q= \dim H^1(S,\mathcal O_S)$.
    \begin{theorem}{(Classification theorem)}\label {3.1}\\If $\kappa(S)=0$, then $S$ is one of the
    following surfaces:
    \begin{enumerate}[1)]\item $K3$ surface, $p_g=1, q=0, K_S=0$.
    \item Enriques surface, $p_g=0, q=0, 2K_S=0$.
    \item Bielliptic surface.\\This means there are two elliptic curves $E$, $F$ and a finite
        group $G$  of translations of $E$ acting also on $F$ such that $F/G\simeq \mathbb P^1$,
        and $S\simeq (E\times F)/G$. In this case, $p_g=0$, $q=1$.
        \item Abelian surface. In this case $p_g=1$, $q=2.$ \end{enumerate}\end{theorem}
        \begin{proof}~\cite {Beauville}, Chapter 8, Theorem 2.\end{proof}
Observe that in cases 1) and 2), $q=0$; in cases 3) and 4), $q>0$. We consider the case $q>0$ first,
then in the second part of this section, we consider the case $q=0$.
        \subsection{The case $q>0$}First let us recall the classification  of bielliptic surfaces in
        the following lemma. Our main results of this subsection are Theorem 3.3 and Theorem 3.4.
\begin{lemma}Keep the notation of Theorem 3.1, then every bielliptic surface is of one of the
following types:
\begin{enumerate}[1)]
\item $(E\times F)/G$, $G=\mathbb Z/{2\mathbb Z}$, acting on $F$ by $x\mapsto -x$.
\item $(E\times F)/G$, $G=\mathbb Z/{2\mathbb Z}\oplus \mathbb Z/{2\mathbb Z}$, acting on $F$ by
    $x\mapsto -x, x\mapsto x+\varepsilon$, where $\varepsilon\in \mathbb F_2$.
\item $(E\times F_i)/G$, where $F_i=\mathbb C/(\mathbb Z\oplus \mathbb Zi)$, and $G=\mathbb
    Z/{4\mathbb Z}$, acting on $F$  by $ x \mapsto i x$.
\item $(E\times F_i)/G$,  $G=\mathbb Z/{4\mathbb Z}\oplus \mathbb Z/{2\mathbb Z}$, acting on $F$
    by $x\mapsto ix,\quad x\mapsto x+\frac{1+i}{2}$.
\item  $(E\times F_{\rho})/G$, where $F_{\rho}=\mathbb C/({\mathbb Z\oplus \mathbb Z\rho})$, $\rho
    $ is a primitive cubic root of identity,  $G=\mathbb Z/{3\mathbb Z}$ acting on $F$ by
    $x\mapsto \rho x$.
\item $(E\times F_{\rho})/G$, and $G=\mathbb Z/{3\mathbb Z}\oplus \mathbb Z/{3\mathbb Z}$, acting
    on $F$ by $x\mapsto \rho x,\quad x\mapsto x+\frac{1-\rho}{3} $.
\item$(E\times F_{\rho})/G$, and $G=\mathbb Z/{6\mathbb Z}$, acting on $F$ by $x\mapsto -\rho
    x$.\end{enumerate}\end{lemma}
\begin{proof}~\cite{Beauville}, Chapter 6, Proposition 6.20.\end{proof}
\begin{theorem}\label{M2}Keep the notation of Lemma 3.2. \\If $S$ is bielliptic, and
\xymatrix{S\ar[r]|\pi &C\ar@/_/@{>}[l]|\sigma}, then:
 \begin{enumerate}[a)] \item $E/G\simeq C$.\item the action of $G$ on $F$ has a fixed point
 $P$.\item
$S$ is of type 1), 3), 5) or 7). Moreover, $\pi=p_1$ is the first projection and  $\sigma:
E/G\longrightarrow S\simeq (E\times F)/G$ is of the form $\sigma(x)=(x,P)$.
\end{enumerate} \end{theorem}
\begin{theorem}\label{M3}Let $S$ be an abelian surface, and  \xymatrix{S\ar[r]|\pi
&C\ar@/_/@{>}[l]|\sigma} . Then $S\simeq C\times E$ for some elliptic curve $E$, $\pi=pr_1$ is the
first projection and $\sigma: C\longrightarrow C\times E$, is defined by $x\longmapsto (x,y_0)$ for
some $y_0\in E$.\end{theorem}
In the following, we analyze the Albanese variety $Alb(S)$ and show that $C$ is elliptic. Then we
use this result to prove Theorem 3.4. For Theorem 3.3, we transfer our problem to the existence of
$G$-equivariant morphisms from $E$ to $F$. Then by some detailed calculations, we find that the
action of $G$ on $F$ must have a fixed point.\\
In the following lemma, we show that $C$ is elliptic.\\Before starting the proof, we recall the
definition and the universal property of Albanese variety: ( see Remark 14, Chapter 5, ~\cite
{Beauville} )\\
Let $X$ be a smooth projective variety. There exists an abelian variety A and a morphism
$\alpha_X:X\longrightarrow A$ with the following universal property:
   for any complex torus $T$ and any morphism $f:X\longrightarrow T$, there exists a unique morphism
   $\tilde f: A\longrightarrow T$, such that $f\circ \alpha=f$. The abelian variety A is called the
   Albanese variety of $X$, and written by $Alb(X)$; the morphism $\alpha_X$ is called the Albanese
   morphism.

        \begin{lemma} If $q(S)>0$ and there exists a curve $C$ s.t. \xymatrix{S\ar[r]|\pi
        &C\ar@/_/@{>}[l]|\sigma}, then $C$ must be elliptic. \end{lemma}
        \begin{proof}By Theorem 3.1, $S$ is bielliptic or abelian.\\
        a) The case when $S$ is abelian.  Firstly, $C$ is not rational, because an abelian variety
        contains no rational curve.
      Now we prove $g(C)\le 1$. \\Recall that for an arbitrary smooth variety $X$, $Alb(X)$, as a
      group, is generated by the image $\alpha(X)$ . \\So the surjective morphism
      $\pi:S\longrightarrow C$ induces a surjective morphism $\tilde \pi :Alb(S)\longrightarrow
      Alb(C)$ such that the following diagram:\\
       $$ \xymatrix {S\ar@{->>}[r]^{\pi}\ar[d]^{\alpha_S}\ar[rd]|{\tilde \pi
       \circ\alpha_S}&C\ar[d]^{\alpha_C}\\Alb(S)\ar@{->>}[r]^{\tilde\pi}&Alb(C)}\\$$is commutative.

Since $\alpha_S$ is an isomorphism, $\tilde \pi \circ\alpha_S$ is surjective. Note that it factors
through the curve $C$, so $g(C)= \dim(Alb(C))\le1$.\\
    We conclude that $C$ must be elliptic.\\
        b) The case $S$ is bielliptic.\\
Using the notation of Theorem 3.1,  we write $S$ as $(E\times F)/G$.\\
     There is a  diagram:\\
       $$\xymatrix{E\times F\ar[r]^{\varphi\quad} & (E\times F)/G \ar[r]^{\qquad \pi} & C}. $$ where
       $\varphi$ is the quotient morphism. Let $f=\pi\circ\varphi$, we get a morphism from the
       abelian surface $E\times F$ to $C$. Then $f$ induces a surjective  morphism $\tilde f:E\times
       F\longrightarrow Alb(C) $, such that  the  diagram:\\
      $$ \xymatrix{E\times F \ar[r]^f \ar@{->>}[d]^{\tilde f}& C \ar[dl]^{\alpha_C} \\ Alb(C)}$$\\is
      commutative.
       Since $f$ factors through the curve $C$, $\dim(Alb(C))\le 1$, so $g(C)\le 1$.\\
       If $g(C)=0$, we consider the cartesian square:\\
      $$ \xymatrix{Y\ar[r]^{\sigma'\quad}\ar[d]&E\times F\ar[d]^{\varphi}\\ C\simeq \mathbb P^1
      \ar[r]^{\sigma\quad}& (E\times F)/G}$$\\ Observe that $G$ acts freely on $E\times F$, the
      quotient morphism $\varphi:E\times F\longrightarrow (E\times F)/G$ is \'{e}tale, so $Y$ is
      also an \'{e}tale cover of  $\mathbb P^1$. Because $\mathbb P^1$ is simply-connected, $Y$ must
      be a disjoint union of finitely many copies of $\mathbb P^1$.
       So there is a closed immersion $\sigma': \coprod \mathbb P^1\longrightarrow E\times F$. Since
       $E\times F$ is an abelian surface, and there are no rational curves on it, we get a
       contradiction. So $C$ is an elliptic curve and this completes our proof.

        \end{proof}Now let us complete the proof of Theorem 3.4.
\begin{proof} By the previous lemma, $C$ is an elliptic curve. Pick a point $P_0$ on $C$ as the
origin, then $C$ becomes a  1-dimensional algebraic group. So $\pi$ is a homomorphism of abelian
varieties. Let $E=ker(\pi)$, then $S\simeq C\times E$.\end{proof}
        In the rest of this part, we assume $S$ is bielliptic. \begin{lemma}If
        $\xymatrix{S\ar[r]|\pi &C\ar@/_/@{>}[l]|\sigma}$, and $g(C)=1$, then $C\simeq Alb(S)$ and
        $\alpha_C\circ \pi=\alpha_S$. \end{lemma}
\begin{proof}
 By the universal property of the Albanese variety, $\pi$ and $\varphi$ induce the diagram:
 $$\xymatrix{Alb(S)\ar[r]|{\pi' }&Alb(C)\ar@/_/@{>}[l]|{\sigma'}}$$ \\such that $ \pi '  \circ
 \sigma '=id_{Alb(C)}$. So $Alb(S)\simeq Alb(C)\times ker(\pi').$ Since $C $ is elliptic, $\dim
 Alb(C)=1$ and $ Alb(C)\simeq C  $. Observe that $q(S)=\dim Alb(S)=1$, so $\dim Alb(S)=\dim Alb(C)$.
 Then $Alb(S)\simeq Alb(C)\simeq C$. \end{proof}

\begin{lemma} Using the notations of Theorem 3.1, we  write $S$ as $(E\times F)/G$. If
\xymatrix{S\ar[r]|\pi &C\ar@/_/@{>}[l]|\sigma}, then $C\simeq E/G$ and $\pi=p_1$. \end{lemma}
\begin{proof}By Lemma 3.5, $C$ is elliptic, then $C$ satisfies all the assumptions of Lemma 3.6, so
$C\simeq Alb(S)$. \\
Recall that $G$ is a group of translations of  $E$. Also $E/G$ is elliptic, so that $Alb(E/G)=E/G$.
By the universal property of the Albanese variety,  the first projection $p_1$  induces a morphism
$\widetilde {p_1}$ such that the diagram :\\
$$\xymatrix{S \ar[r]^{p_1}\ar[d]_{\pi}&E/G\\C\ar[ur]_{\widetilde {p_1 }}}$$ is commutative.\\ Note
that $\widetilde {p_1}$ is a finite morphism between projective curves and $p_1$ has connected
fibres, so $deg(\widetilde{p_1})=1$. Then $\widetilde{p_1}$ is an isomorphism, which completes our
proof. \end{proof}
In the following lemma, we transfer our problem to the existence of $G$- morphisms from $E$ to $F$.
\begin{lemma}   There exists a section $\sigma$ of $p_1:S\longrightarrow E/G$ if and only if there
exists a $G$-morphism $h:E\longrightarrow F$.\end{lemma}
\begin{proof} The ``if'' part is trivial, we now prove the ``only if'' part.\\
 The quotient morphism $\varphi: E\longrightarrow E/G$ induces a  cartesian square:\\
$$\xymatrix{E\times F\ar[r]^{\tilde\varphi}\ar[d]^{p_1'}&(E\times
F)/G\ar[d]^{p_1}\\E\ar[r]^{\varphi}&E/G}$$where $p_1'$ is the first projection, and $\tilde\varphi$
is the quotient morphisms. \\
By the above cartesian square, any section $\sigma$ of $p_1$  induces a section of $p_1'$. We denote
it by  $\delta:E\longrightarrow E\times F$, for all $ x\in E$, $\delta(x)=(x,h(x))$, where
$h=p_2\circ\delta$.\\
Now we illustrate all the morphisms in the following commutative
diagram:\\$$\xymatrix{E\ar[dr]|{\delta}\ar[drr]|{\sigma\circ\varphi}\ar[ddr]|{id_E}\\ & E\times F
\ar[r]^{\tilde\varphi\quad} \ar[d]|{p_1'}& (E\times F)/G\ar[d]|{p_1}\\ &E\ar[r]^{\varphi\quad} & E/G
\ar@/_/@{>}[u]_{\sigma}}$$\\and verify that $h$ is a $G$-morphism.
For any $x\in E$ and $ g\in G$, we denote the action of $g $ on $x$, by $g\cdot x$. We aim to show
that $\delta(g\cdot x)=g\cdot\delta(x)$. It surffice to verify that \begin{equation}
\sigma\circ\varphi(g\cdot x)=g\cdot(\sigma\circ\varphi(x)).\end{equation}Since $G$ acts on $(E\times
F)/G$ trivially, \begin{equation} g\cdot(\sigma\circ\varphi(x))=\sigma\circ\varphi(x).\end{equation}
Observe that $\varphi$ is a quotient morphism, so \begin{equation} \varphi(g\cdot
x)=\varphi(x).\end{equation} So \begin{equation} \sigma\circ\varphi(g\cdot
x)=\sigma\circ\varphi(x)=g\cdot(\sigma\circ\varphi(x)).\end{equation} Equation (21) holds, so
$\delta$ is a $G$-morphism, hence so is $h$. \end{proof}

Now we show that the action of $G$ on $F$ has a fixed point.
\begin{lemma}If $p_1: S\longrightarrow E/G$ has a section $\sigma$, then the action of $G$ on $F$
has a fixed point.\end{lemma}
\begin{proof} By Lemma 3.8,  $\sigma$ induces a $G$-morphism $h: E\longrightarrow F$. We write it as
for all $ x\in E$,  $h(x)=Ax+a$,  where $A$ is the linear part of $h$. Observe that  $G$ acts on $E$
as translations, we denote its action  by: for all $ x\in E$, $g\cdot x= x+x_g$. And we denote the
action of $G$ on $F$ by:  for all $  y \in F$, $g\cdot y=l(y)+ \varepsilon _g$, where $l$ is the
linear part.\\
Since $h(g\cdot x)=g\cdot h(x)$, we have \begin{equation}h(x+x_g)=g\cdot (Ax+a).\end{equation}  The
left hand  of (25) equals \begin{equation}h(x+x_g)=A(x+x_g)+a.\end{equation} The right hand of (25)
equals \begin{align}&g\cdot(Ax+a)\\=&l(Ax+a)+\varepsilon _g\\=&(l(Ax)+\varepsilon
_g)+(l(a)+\varepsilon _g)-\varepsilon _g\\=&g\cdot (Ax)+g\cdot a-\varepsilon _g.\end{align}
So \begin{equation}A(x+x_g)+a=g\cdot (Ax)+g\cdot a-\varepsilon _g.\end{equation}
Let $x=0$ in the above equation (31), we have \begin{equation}A(x_g)+a=(g\cdot 0-\varepsilon
_g)+g\cdot a.\end{equation} But \begin{equation}g\cdot 0=l(0)+\varepsilon _g=\varepsilon
_g.\end{equation} So \begin{equation}A(x_g)+a=g\cdot a. \end{equation} Subtracting equation (34)
from equation (31), we get \begin{equation}Ax=g\cdot Ax-\varepsilon _g, \end{equation}  which
implies that \begin{equation}g\cdot Ax=Ax+\varepsilon _g.\end{equation}\\
If $h:E\longrightarrow F$ is constant, i.e. $h$ maps $E$ to a single point $P\in F$, then $h$ is
$G-$morphism, i.e. $g\cdot h(x)=h(g\cdot x) $, implies that $g\cdot P= P$, which means the action of
$G$ on $F$ has a fixed point.\\
Otherwise $h$ is surjective. Then its linear part $A$, is a linear automorphism of the complex plane
which maps the lattice of $E$ into the lattice of $F$. So by equation (36), for all $ y \in F$, we
have $g\cdot y=y+\varepsilon _g $, which means that $G$ acts on $F$ by translations. \\
Note that for all  types of bielliptic surfaces in Lemma 3.2, the action of $G$ on $F$ has a
non-trivial linear part. So $h$ is not surjective, hence it is a constant morphism, and the action
of $G$ on $F$ has a fixed point. \\
\end{proof}
Observe that, for surfaces of types $2)$, $4)$, $6)$ in Lemma 3.2, $G$ has no fixed point on $F$.
This completes the proof of Theorem 3.3.
\\Now recall a  known result (see  ~\cite {Michel. Brion}, Section 4.5,  {Remark 16}):
 Consider the functor of composition laws on a variety $S$, i.e. the contravariant functor from
 schemes to sets given by $T\longmapsto Hom(S\times S \times T, S)$, then the families of algebraic
 semigroup laws yield a closed subfunctor, and this subfunctor is represented by a closed subscheme
 $SL(S)\subseteq Hom(S\times S, S)$. \\For a given algebraic law $\mu_{t_0}$, we denote its
 associated contraction by $\xymatrix{S\ar[r]|\pi &C\ar@/_/@{>}[l]|\sigma}$ and its associated
 abelian variety by $A$, then the connected component of $\mu_{t_0}$ in $SL(S)$ is identified with
 the closed subscheme of $Hom(C, S)\times A$ consisting of those pairs $(\varphi, g)$ such that
 $\varphi$ is a section of $\pi : S\longrightarrow C$. We denote the scheme of sections of $\pi$ by
 $Mor_{\pi}(C, S)$, it is isomorphic to an open subscheme of $Hilb(S)$ by assigning  every section
 to its image in $S$. By using the local study of $Hilb(S)$, for any section $\sigma$ of $\pi$, the
 dimension of the tangent space of $Hilb(S)$ at $\sigma(C)$ is $h^0(\sigma(C), \mathcal
 N_{\sigma(C)/S})$, and the obstruction lies in $H^1(\sigma(C), \mathcal N_{\sigma(C)/S})$ (where
 $\mathcal N_{\sigma(C)/S}$ is the normal bundle of $\sigma (C)$ in $S$).
 (For the discussion of local study of $Hilb(S)$, see ~\cite{Mumford2}.)\\In the following, we want
 to study the structure of $Mor_{\pi}(C,S)$.
For $Kod(S)=0$, we have the following two theorems.
 \begin{theorem} Assume that $S$ is a bielliptic surface, if there is a curve $C$ satisfying
 \xymatrix{S\ar[r]|\pi &C\ar@/_/@{>}[l]|\sigma}, then $Mor_{\pi}(C, S)$ consists of reduced isolated
 points. \end{theorem}

\begin{theorem} If $S$ is an abelian surface, and there is a curve $E$ satisfying
\xymatrix{S\ar[r]|\pi & E\ar@/_/@{>}[l]|\sigma}, then $Mor_{\pi}(E,S)= S/E =ker (\pi)$.\end{theorem}
We postpone the proof of Theorem 3.10 and Theorem 3.11 to Section $6$. There we will determine the
structure of $Mor_{\pi}(C,S)$ for any smooth elliptic fibration $\pi$ in a more uniform way, see
Theorem 6.1. Then Theorem 3.10 and Theorem 3.11 are just direct corollaries of Theorem 6.1.

\subsection{The case $q(S)=0$}{The main result of this part is Theorem \ref {M4}.
\begin{lemma}If $q(S)=0$, and there exists \xymatrix{S\ar[r]|\pi &C\ar@/_/@{>}[l]|\sigma} , then $C$
is rational and  $\pi$ is an elliptic fibration.\end{lemma}
\begin{proof} If $q(S)=0$, then by Theorem \ref {3.1},  $S$ is $K3$ or Enriques. In both cases,
$K_S\sim_{num}0$. Pick an arbitrary general fibre $F$ of $\pi$. By the genus formula, we have:
\begin{equation}g(F)=1+\frac{1}{2}(F\cdot K_S+ F^2)=1\end{equation} so $\pi$ is an elliptic
fibration. On the other hand, $\pi$ induces a surjective morphism $\tilde\pi: Alb(S)\longrightarrow
Alb(C)$. So \begin{equation}g(C)=dim Alb(C)\le dim Alb(S)=q(S)=0.\end{equation}Then $C$ is
rational. \end{proof}
\begin{lemma}Every elliptic fibration $\pi:S\longrightarrow \mathbb P^1$ of an Enriques surface $S$
has exactly two multiple fibres, $2F$ and $2F'$. \end{lemma}
\begin{proof}~\cite {BPV}, Chapter 8, Lemma 17.2.\end{proof}
\begin{theorem}\label {M4}If $S$ is an Enriques surface or a general $K3$ surface, there is no
\xymatrix{S\ar[r]|\pi &C\ar@/_/@{>}[l]|\sigma} .(Hence there is no non-trivial semigroup structure
on $S$.)\end{theorem}
\begin{proof}If $S$ is an Enriques surface and  such a diagram exist, then $C$ is rational,  and
$\pi$ is an elliptic fibration. By Lemma 3.11, there exists a multiple fibre $2F$. For a general
smooth fibre $F_0$,  $\sigma(C)\cdot F_0=1$, but under our assumptions, $\sigma(C)\cdot
F_0=\sigma(C)\cdot 2 F\ge2$, which is a contradiction. \\
The case for  $S$ is a general $K3$ surface, see Proposition 7.1.3, ~\cite{Daniel Huybrechts}
\end{proof}}
\begin{theorem} For any  $K3$ surface $S$ and any fibration $\pi:S\longrightarrow C$,
$Mor_{\pi}(C,S)$ consists of reduced isolated points.
 \end{theorem}\begin{proof}Since $q(S)=0$, $C\simeq \mathbb P^1$. Let us determine the structure of
 $Mor_{\pi}(\mathbb P^1, S)$. For any section $\sigma$ of $\pi$, the tangent space of $Hilb(S)$ at
 $\sigma(\mathbb P^1)$ is $H^0(\mathbb P^1, \mathcal N_{\sigma(\mathbb P^1)/S})$, where $ \mathcal
 N_{\sigma(\mathbb P^1)/S})$ is the normal bundle of $\sigma(\mathbb P^1)$ in $S$. Since  $\mathbb
 P^1$ and $S$ are both smooth, by the adjunction formula,  $\mathcal N_{\sigma(\mathbb P^1)/S} =
 \omega^{-1}_S\mid _{\sigma(\mathbb P^1)}\otimes \omega _{\mathbb P^1}$. By the definition of $K3$,
 $\omega_S$  is trivial, so $ \mathcal N_{\sigma(\mathbb P^1)/S}$ is $\mathcal O(-2)$, and $\dim
 H^0(\mathbb P^1, \mathcal N_{\sigma(\mathbb P^1)/S})=0$. So $ \dim T_{\sigma(\mathbb P^1)}(Hilb
 (S))=0$, so the dimension of $Mor_{\pi}(\mathbb P^1, S)$ at $\sigma(\mathbb P^1)$ is zero.
 \end{proof}
\section{The case $\kappa(S)=1$}{
In this section, we always assume $\kappa(S)=1$. The main result  is Theorem 4.2.
\begin{lemma} There exists an elliptic fibration $\varphi: S\longrightarrow B$, where $B$ is a
smooth projective curve.\end{lemma}
\begin{proof}~\cite {Beauville}, Chapter 9, Proposition 2.\end{proof}
So we have the following diagram:\\
$$\xymatrix{S\ar[r]|\pi\ar[d]_{\varphi} &C\ar@/_/@{>}[l]|\sigma\\B}.$$
\textbf{Remark 2):} If $\pi=\varphi $, then $\varphi$ is an elliptic fibration with a section. This
implies that the generic fibre $F$ of $\varphi$ is a smooth curve of genus 1 over the functional
field of $B$, and the set of its rational points is nonempty. So $F$ is an elliptic curve and it can
be imbedded into $\mathbb P^2_{k(B)}$ as a cubic curve. For any two minimal surfaces $S_i$ and two
elliptic fibrations $\varphi _i: S_i\longrightarrow B$, if their generic fibres are isomorphic, then
the two fibrations $\varphi _i$ are isomorphic. So to give a minimal surface $S$ and an elliptic
fibration  $\varphi: S\longrightarrow B$ is equivalent to give a homogeneous cubic equation in
$\mathbb P^2_{k(B)}$. So in what follows, we focus on the case $\pi \ne \varphi$.
\begin{theorem}\label{M5}If $g(C)\ge 1$, then $S\simeq B\times C$ and $\varphi$ is the first
projection, $\pi$ is the second projection. \end{theorem}
The key point to prove Theorem 4.2 is to show that $\sigma $ maps $C$ into a fibre of $\varphi$.
First we want to determine the singular fibres of the elliptic fibration $\varphi$. For this, we
recall ``Kodaira's table of singular fibres'' .
\begin{lemma}{(Kodaira's table of singular fibres)}\\
Let $f:S\longrightarrow \Delta$ be an elliptic fibration over the unit disk $\Delta$, such that all
fibres $S_x, x\ne0$, are smooth. We  list all possibilities for the central fibre $S_0$:
\begin{enumerate}[a)]
\item $S_0$ is irreducible. Then $S_0$ is either smooth elliptic, or rational with a node, or
    rational with a cusp.
\item $S_0$ is reducible. Then  every component $C_i$ of $S_0=\sum n_iC_i$ is a ($-2$)- rational
    curve.
\item $S_0$ is a multiple fibre, we write it as $S_0=mS_0'$. Then $S_0'$ is smooth elliptic, or
    rational with a node, or of type b) \end{enumerate}\end{lemma}
\begin{proof} ~\cite {BPV}, Chapter 5, Proposition 8.1.\end{proof}
\begin{lemma}If $g(C)\ge1$, then $g(C)$ must be $1$. Moreover any fibre of $\varphi$ is a multiple
of a smooth elliptic curve.\end{lemma}
\begin{proof} Assume $g(C)\ge 2$. For an arbitrary smooth fibre $F$ of $\varphi$, \begin{equation}
1=g(F)<g(C)\end{equation} so  $\pi$ contracts $F$. By the ``Rigidity lemma'', $\pi$ will factor
through $\varphi$, which contradicts  our assumptions. So $g(C)=1$. \\ Let  $F_0$ be a singular
fibre. If $F_0 $is not a multiple of a smooth elliptic curve, then by ``Kodaira's table of singular
fibres'', $F_0$ is a sum of irreducible rational curves. Then $F_0$ will be contracted by $\pi$, and
by the ``Rigidity lemma'' again, $\pi$ will factor through $\varphi$. This contradicts our
assumptions.
So every fibre of $\varphi$ is a smooth elliptic curve, or a multiple of a smooth elliptic curve.
\end{proof}
In view of Lemma 4.4, the elliptic fibration $\varphi$ is ``almost smooth'': its singular fibres
are only multiples of smooth curves. In the following lemma, we will see that after performing a
base change, we can eliminate all the multiple fibres, and get a  smooth elliptic fibration.
\begin{lemma}Let $\varphi: S\longrightarrow B$ be a morphism from a surface onto a smooth curve
whose fibres are multiples of smooth curves. Then there are a ramified Galois cover
$q:B'\longrightarrow B$ with Galois group $G$, a surface $S'$ and a commutative cartesian square: \\
$$\xymatrix{S'\ar[r]^{q'}\ar[d]^{p'}&S\ar[d]^{p}\\B'\ar[r]^q &B}$$ such that the action of $G$ on
$B'$ lifts to $S'$, $q'$ induces an isomorphism $S'/G\simeq S$ and $p'$ is smooth.\end{lemma}
\begin{proof} Just use the cyclic covering trick, ~\cite {Beauville}, Chapter 6, Lemma 7.\end{proof}
In some cases, we can even get a trivial elliptic fibration, after performing successive  base
changes.
\begin{lemma}Let $p:S\longrightarrow B$ be a smooth morphism from a surface to a curve, and $F$ a
fibre o f $p$. Assume either that $g(B)=1$ and $g(F)\ge 1$, or that $g(F)=1$. Then there exists an
\'{e}tale cover $B'$ of $B$, such that the fibration $p':S'=S\times _B B'$ is trivial, i.e.
$S'\simeq B'\times F$. Furthermore, we can take the cover $B'\longrightarrow B$ to be Galois with
group $G$, say, so that $S\simeq (B'\times F)/G$.\end{lemma}
\begin{proof} ~\cite {Beauville}, Chapter 6, Proposition 8.\end{proof}
\begin{lemma} Assume that $S$ and $C$ satisfy all the assumptions of Theorem 4.2.
Then $\sigma $ maps $C$ into a fibre of $\varphi$.\end{lemma}
\begin{proof} If not, $\varphi\circ\sigma:C\longrightarrow B$ will be a finite morphism between
curves. By Hurwitz's Theorem, $g(B)\le 1$. Let
$\varphi_C=\varphi|_{\sigma(C)}:\sigma(C)\longrightarrow B$, we now determine the degree of the
ramification divisor $R$ of $\varphi_C$. \\
If $g(B)=1$, assume that $\varphi$ has a multiple fibre $F_b$, and write $F_b=mF_b'$. Let
$\sigma(C)$ intersect with $F_b'$ at a point $P$, then the ramification index of $\varphi_C$ at $P$
satisfies \begin{equation}e_P\ge m>1.\end{equation} Then we get a ramified morphism $\varphi_C$
between elliptic curves, which is impossible.\\So $\varphi: S\longrightarrow B$ is a smooth elliptic
fibration.\\
Then by Lemma 4.6, there exists an  \'{e}tale cover $B'\longrightarrow B$ such that $S'=B'\times_BS$
is a trivial elliptic fibration.
 Then\begin{equation}Kod(S')=Kod(B'\times F)=Kod(B')+Kod(F)\end{equation} where $F$ is a general
 fibre of $\varphi$.  Since $B'$ and $F$ are both elliptic curves, $Kod(S')=0$. But  $Kod(S')\ge
 Kod(S) \ge 1$ , which  is a contradiction. \\
 So $g(B)=0$.
 Assume that  at points $b_1,\cdots,b_s$ of $B$, $\varphi$ has multiple fibres $F_i=m_iF_i'$ ($1\le
 i\le s$) and each $F_i'$ intersects with $\sigma (C)$  at points $P_{i,1},\cdots ,P_{i,i_j}$ with
 multiplicities $n_{i,1},\cdots, n_{i,j_i}$. Then $deg(R)$ saisfies:
\begin {align}
 deg(R)&\ge \sum_{i=1}^{s}\sum_{k=1}^{j_i}(e_{P_{i,k}}-1)\\ & =
 \sum_{i=1}^{s}\sum_{k=1}^{j_i}(m_in_{i,k}-1)\\&=\sum_{i=1}^{s}\{m_i(\sum_{k=1}^{j_i}n_{i,k})-j_i\}\\
&=\sum_{i=1}^{s}(m_iF_i'\cdot
\sigma(C)-j_i)\\&=\sum_{i=1}^{s}(F_i\cdot\sigma(C)-j_i)\\&=\sum_{i=1}^{s}(deg(\varphi_C)-j_i)\end
{align}
Since $\sum_{k=1}^{j_i}n_{i,j}=F_i'\cdot \sigma(C)$ and  each $n_{i,j}\ge 1$, so \begin{equation}
j_i\le F_i'\cdot \sigma(C)=deg(\varphi_C)/m_i.\end {equation} So \begin{equation}
deg(R)\ge\sum_{i=1}^{s}(deg(\varphi_C)-j_i)\ge \sum_{i=1}^{s}deg(\varphi_C)(1-\frac{1}{m_i}).\end
{equation} Using Hurwitz's Theorem, we  calculate $deg(R)$ as:\begin{align}
&deg(R)\\=&2g(E)-2-deg(\varphi_C)(2g(B)-2)\\=&2deg(\varphi_C).\end{align} Now by comparing (49) and
(52), we get \begin{equation}\sum_{i=1}^{s}(1-\frac{1}{m_i})-2\le0.\end {equation}

Recall Lemma 7.1 in ~\cite {Wall}: for any elliptic fibration $\varphi$, we define an invariant by
\begin{equation}
\delta(\varphi)=\chi(\mathcal O_S)+2g(B)-2+\sum_{i=1}^{s}(1-\frac{1}{m_i}),\end {equation} then
$Kod(S)=1$ if and only if \begin{equation}\delta(\varphi)>0 .\end {equation}In our situation:
\begin{equation}\delta(\varphi)=\chi(\mathcal O_S)-2+\sum_{i=1}^{s}(1-\frac{1}{m_i}).\end {equation}
By the inequality (53),  $\delta(\varphi)\le \chi(\mathcal O_S) $.\\
Now let us  determine $\chi(\mathcal O_S)$. Recall  Noether's Formula:\begin{equation}\chi(\mathcal
O_S)=\frac{1}{12}(K_S^2+\chi_{top}(S)).\end {equation}\\Since $S$ is minimal and $Kod(S)=1$, we have
$K_S^2=0$. So \begin{equation}\chi(\mathcal O_S)=\frac{1}{12}\chi_{top}(S).\end {equation}
Recall that
\begin{equation}\chi_{top}(S)=\chi_{top}(B)\chi_{top}(F)+\sum_{i=1}^{s}(\chi_{top}(F_{i})-\chi_{top}(F)),
\end{equation} where $F_{i}$ are the singular fibres and $F$ is a general smooth fibre. In our case,
$F_i$ is a multiple of a smooth elliptic curve. So $\chi_{top}(F_{i})=\chi_{top}(F)=0$, hence
$\chi_{top}(S)=0$, which implies \begin{equation}\delta(\varphi)\le 0.\end {equation} This
contradicts the assumption that  $Kod(S)=1$. \\
In conclusion,  the assumption that $\varphi\circ\sigma:C\longrightarrow B$ is  a finite morphism
always leads to contradictions, so $\sigma$ maps $E$ into a fibre of $\varphi$.
\end{proof}
Now we begin the proof of Theorem 4.2.
\begin{proof}  We first consider the case when $\varphi$ is smooth.\\First we define
$\alpha=(\pi,\varphi):S\longrightarrow C\times B$, we show that $\alpha$ is an isomorphism. By Lemma
4.7, $\sigma$ maps $C$ into a fibre  $F_b$ of $\varphi$. If  $\varphi$ is smooth, $F_b$ is integral,
so  $\sigma (C)=F_{b}$. For an arbitrary  fibre $F$ of $\pi$, since $\sigma$ is a section of $\pi$ ,
$F\cdot F_{b}=F\cdot\sigma(C)=1$. Then $\alpha$  is birational. If  there exists an irreducible
curve $X$ on $S$ contracted by $\alpha$, then $X$ will be a fibre of $\varphi$, because $\varphi$
has smooth integral fibres. So $\pi$ contracts one fibre of $\varphi$. By the ``Rigidity lemma'',
$\pi$ will factor through $\varphi$, contradicts to our assumptions. So $\alpha$ is a birational
quasi-finite, projective morphism, which is certainly an isomorphism. \\Now we deal with the general
case, which allows $\varphi$ has singular fibres. We know that, by Lemma 4.5,  there exists
$q:B'\longrightarrow B$, such that $\varphi' :S'\simeq B'\times_B S\longrightarrow B'$ is smooth.
Note that $\varphi\circ \sigma:C\longrightarrow B$ maps $C$ to a single point of $b\in B$.  Then
there exists a constant morphism $f:C\longrightarrow B'$ lifts $\varphi\circ\sigma$, such that  the
diagram\\
$$\xymatrix{&B'\ar[d]^q\\C\ar[r]^{\varphi\circ\sigma}\ar@{.>}[ur]^f&B}$$ is commutative.
By the universal property of the fibre product,
  there exists a morphism $\sigma':C\longrightarrow S'$ satisfying $q'\circ \sigma'=\sigma$ and
  $\varphi'\circ q=f$. \\We  illustrate all the morphisms  in the following diagram:\\
$$\xymatrix{S'\ar[r]_{q'}\ar[d]_{\varphi'}&S\ar[d]^{\varphi}\ar[r]|\pi &C\ar@/^/@{>}[l]|\sigma
\ar@/_/@{>}[ll]_{\sigma'}\\B'\ar[r]^q&B}$$
 and verify $\sigma'$ is a section of  $\pi\circ q'$:\begin{equation}(\pi\circ
 q')\circ\sigma'=\pi\circ(q' \circ\sigma')=\pi\circ\sigma'=id_C.\end{equation}
 So there is a  diagram:\\
$$\xymatrix{S'\ar[r]|{\pi'}\ar[d]_{\varphi'} &C\ar@/_/@{>}[l]|{\sigma'}\ar[dl]^f\\B'}.$$ \\
such that \begin{equation}
\pi'=\pi\circ q'\end{equation}\begin{equation}\varphi'\circ \sigma'=f.\end{equation}  \\Then
$\varphi'\circ \sigma'(C)=f(C)=\{b'\}$, which means $\sigma'$ maps $C$ into a fibre $F_{b'}$ of
$\varphi'$. Now $S'$ is a smooth fibration over $C$ and $\sigma'$ maps $C$ into a fibre $F_{b'}$ of
$\varphi'$. \\
According to  what we have discussed about our problem in  the smooth case,
$$\alpha'=(\pi',\varphi'):S'\longrightarrow C\times B'$$ is an isomorphism.\\
Let $G$ be the Galois group of the covering $q:B'\longrightarrow B$, then $S\simeq S'/G\simeq
(C\times B')/G$ and the diagram:\\
$$\xymatrix{S'\simeq C\times B'\ar[d]^{q'}\ar[dr]^{p_1}\\S\simeq (C\times B')/G
\ar[r]^{\qquad\quad\pi}&C}$$\\is commutative.
 We denote the action of $G$ on $C\times B'$ by \begin{equation}
g(x,b')=\big(\phi_{b'}(g)(x), g b'\big).\end{equation} Since $p_1$ factors through the quotient
morphism $q'$, \begin{equation}p_1\big(g(x,b')\big)=p_1\big((x,b')\big).\end{equation} This means
that for all $ g\in G$ and for all $ x\in C$, $\phi_{b'}(g)(x)=x$. So $G$ acts trivially on the
factor $C$. Then we have $$S\simeq (C\times B')/G\simeq C\times(B'/G)\simeq C\times B.$$
 The proof of Theorem 4.2 is completed.\end{proof}
Now we prove a result about sections of non-smooth elliptic fibrations.
\begin{theorem} If there is a curve $C$ satisfies that \xymatrix{S\ar[r]|\pi
&C\ar@/_/@{>}[l]|\sigma}, and $\pi$ is not a smooth elliptic fibration, then $Mor_{\pi}(C,S)$
consists of reduced isolated points.\end{theorem}
\begin{proof}Consider the diagram we introduced in Lemma 4.1,
$$\xymatrix{S\ar[r]|\pi\ar[d]_{\varphi} &C\ar@/_/@{>}[l]|\sigma\\B}.$$ Then there are three
possibilities:
\begin{enumerate}[1)]\item $(B,\varphi)\simeq (C,\pi).$\item $(B,\varphi)\not\simeq (C,\pi)$ and
$g(C)\ge 1$.\item $(B,\varphi)\not\simeq (C,\pi)$ and $g(C)=0$.\end{enumerate}
For case 2), by Theorem 4.2, $S\simeq B\times C$ and $\pi$ is the second projection, hence it is a
smooth elliptic fibration, contradicts to our assumption. \\For case 3), $C\simeq \mathbb P^1$. By
the ``adjunction formula'', $$0=g(C)=1+\frac{1}{2}(K_S\cdot C+C^2).$$ Since $S$ is an elliptic
surface over $B$ and $Kod(S)=1$, $K_S \sim mF$, where $F$ is a fibre of $\varphi$ and $m$ is
positive. So $K_S$ is nef, which implies that $K_S\cdot C\ge 0$ and $C^2<0$. So $deg(\mathcal
N_{C/S})=C^2<0$, which implies that $\dim H^0(C,\mathcal N_{C/S})=0$.\\
For case 1), $\pi$ is an elliptic fibration over $C$ with a section. \\Then $$K_S\sim^{num}
(\chi(\mathcal O_S)+2g(C)-2)F,$$ where $F$ is a fibre of $\pi$.\\ So $$K_S^{-1}\cdot C=
-(\chi(\mathcal O_S)+2g(C)-2),$$ $$deg(\mathcal N_{C/S})=deg (K_S^{-1}\arrowvert _C\otimes \omega
_C)=-\chi(\mathcal O_S).$$By Noether's formula, \begin{equation}\chi(\mathcal
O_S)=\frac{1}{12}(K_S^2+\chi_{top}(\mathcal O_S)).\end {equation} Since $K_S^2=0$,
\begin{equation}\chi(\mathcal O_S)=\frac{1}{12}\chi_{top}(S).\end {equation}
Recall that
\begin{equation}\chi_{top}(S)=\chi_{top}(B)\chi_{top}(F)+\sum_{i=1}^{s}(\chi_{top}(F_{i})-\chi_{top}(F)),\end
{equation} where $F_{i}$ are the singular fibres and $F$ is a general smooth fibre. If $\pi$ is not
smooth, $\chi_{top}(S)>0$, which implies that $deg(\mathcal N_{C/S}<0)$, so $\dim H^0(C,\mathcal
N_{C/S})=0.$\end{proof}
In the above theorem, we miss the case that $\pi$ is a smooth fibration, we wil give an answer to
this case in Theorem 6.1. Then we can complete our discussion of $Mor_{\pi}(C,S)$, when
$\kappa(S)=1.$
\section{The case $\kappa(S)=2$}
In this section, we always assume that $\kappa(S)=2$. For any non-trivial fibration
$\pi:S\longrightarrow C$, we want to study $Mor_{\pi}(C,S)$ by counting the rational points on the
generic fibre of $\pi$. This idea is illustrated in the following lemma and our main result of this
section is Theorem 5.2.
\begin{lemma}Consider  a fibration $\pi: S\longrightarrow C$. Then sections of $\pi$  are in
one-to-one correspondence with $k(C)$- rational points of the generic fibre. \end{lemma}
\begin{proof}Let $E$ denote the generic fibre of $\pi$ and consider the following cartesian square:
           $$\xymatrix{E \ar[r]^{\pi' \quad} \ar[d]^{\tau}& Spec ~k(C)\ar[d]_{\tau'}\\S
           \ar[r]_{\pi}&C~~,}$$      then any section of $\pi$ induces a section of $\pi'$.
           Conversely, given a section $\sigma'$ of $\pi'$, we pick an arbitrary  point $P\in C$,
           then consider the following commutative diagram:
           $$\xymatrix{Spec ~k(C)\ar[r]^{ \quad\tau \circ \sigma'}\ar[d]_f&S\ar[d]^{\pi}\\Spec
           ~\mathcal O_{C,P}\ar[r]^g&C~~.}$$ Since $\mathcal O_{C,P}$ is a discrete valuation ring
           and $\pi$ is a projective morphism, by the ``Valuative Criterion of Properness'', there
           is a unique morphism $h:Spec \mathcal O_{C,P} \longrightarrow S$ satisfying $h\circ
           f=\tau\circ\sigma'$ and $\pi\circ f =g$. As $P$ varies along $C$, and $C$ is a projective
           smooth curve, we  get a morphism $\sigma: C\longrightarrow S$ satisfying $\pi\circ
           \sigma=id_C$.\end{proof}
\begin{theorem}If there is a fibration $\pi:S\longrightarrow C$ and $S$ is not a product  $C\times
C'$ for any smooth curve $C'$, then there are only finitely many sections of $\pi$.\end{theorem}
\begin{proof} If there are infinitely many sections of $\pi$, we consider the generic fibre $F$ of
$\pi$, then by Lemma 5.1,  there are infinitely many rational points on $F$. Now recall a  theorem
of Manin:(See Theorem 3, ~\cite {Manin}) \\
 Theorem: Let $K$ be a regular extension of the field $k$ of characteristic zero and let $C$ be a
 curve of genus bigger or equal  than $2$ defined over $K$. If the set of points of $C$ defined over
 $K$ is infinite, then there is a curve $C'$ which is birationally equivalent to $C$ over $K$ and
 defined over $k$. All the points of $C_K$ except a finite number are images of points of $C'_{k}$.
 \\ Since $\kappa (S)=2$, the arithmetic genus of every fibre of $\pi$ is  bigger or equal than 2,
  we  can apply this theorem to $F$ and  in our case, $k=\mathbb C$ and $K=k(C)$, then there is a
  curve $C'$ defined over $\mathbb C$ and a birational map $\theta: S\longrightarrow C\times C'$.
  Since $\kappa(S)=2$ and $S$ is minimal, by a corollary of ``Castelnuovo's Theorem'', $\theta$ is
  defined everywhere and an isomorphism (See Theorem 19, Chapter 5, ~\cite {Beauville}). This
  contradicts our assumption. \end{proof}
\section {Sections of elliptic fibration with smooth fibre}
 Now we fix a smooth elliptic fibration $\pi:S\longrightarrow C$ and we want to study
 $Mor_{\pi}(C,S).$ By Lemma 4.6, there is an \'{e}tale cover $D\longrightarrow C$ satisfying that
 $S\times_C D$ is a trivial elliptic fibration. So in what follows, we always assume that $S\simeq
 (D\times E)/G$ and consider the elliptic fibration $\pi : S\longrightarrow D/G.$
\begin {theorem} Assume that $S\simeq (D\times E)/G$ where
\begin{enumerate}[1)]\item $D$ and $E$ are projective smoth curves, $g(D)\ge1$ and $g(E)=1$.
\item $G$ is a finite group, which  acts faithfully on $D$ and $E$, and freely on $D$ and $D\times
    E$. \end{enumerate}Now consider $\pi: S\longrightarrow D/G$, where $\pi$ is the first
    projection. \\Then:
    \begin{enumerate}[1)] \item $Mor_{\pi}(D/G, S)\simeq E/G$  if every element of $G$ acts on $E$
    only as translations.\item $Mor_{\pi}(D/G, S)$ consists of reduced isolated points if some
    element g $(\ne id_G)$  of $G$ has fixed points on $E$.\end{enumerate}\end{theorem}
    \begin{proof}First we prove claim 1). \\We view $Mor_{\pi}(D/G, S)$ as an open subscheme of
    $Hilb(S)$ and by our assumption  for all $ g \in G$ and $ y \in E$, $g\cdot y=y+y_g$, where
    $y_g$ is a constant.  Note that $E$ is an algebraic group,  we want to define its action on
    $Hilb(S)$ and aim to show that $Mor_{\pi}(D/G, S)$, as an open subscheme of $Hilb(S)$, consists
    of only one orbit.
   \\ Now for all $ y_0 \in E$, we define $t_{y_0}: D\times E\longrightarrow D\times E$ as
   $t_{y_0}(x,y)=(x,y+y_0)$. Since $G$ acts on $E$ as translations and $E$ is an abelian group,
   $t_{y_0}$ is a $G-$morphism. So it induces an automorphism $\alpha_{y_0}$ of the quotient
   $(D\times E)/G$, i.e. $\alpha_{y_0}\in Aut (S)$.\\
    Then we define the action of $E$ on $Hilb(S)$ by $$\alpha: E\times Hilb(S)\longrightarrow
    Hilb(S), $$ $\alpha(y_0,P)=\alpha_{y_0}(P)$, where $P$ is any closed subvariety of $S$.\\
    For a section $\sigma$, we want to show that $\pi\circ \alpha_{y_0}\circ\sigma=id_{D/G}$, then
    sections of $\pi$ can move by the action of  $E$. Observe that $\sigma$ induces a $G-$morphism
    $f_{\sigma}:D\longrightarrow E$, so for all $ \bar {x}\in D/G$, $\sigma(\bar {x})=~\overline{(x,
    f_{\sigma}(x))},$ then $$\pi\circ \alpha_{y_0}\circ\sigma(\bar {x})=\pi\circ
    \alpha_{y_0}\overline{(x, f_{\sigma}(x))}=\pi\overline{(x,f_{\sigma(x)}+y_0)}=\bar{x}.$$
    This  implies that the orbit of $\sigma(D/G)$, we denote it by $O$, is contained in
    $Mor_{\pi}(D/G, S)$. \\ Now let us determine the isotropy group of $\sigma(D/G)$. If
    $\alpha_{y_0}\circ \sigma=\sigma$, then $y_0=y_g$ for some $g\in G$. Since $G$ acts faithfully
    on $E$, the isotropy group is isomorphic to $G$ and $O\simeq E/G.$ \\Note that $\pi$ is a smooth
    elliptic fibration with a section $\sigma$. Then $\chi(\mathcal O_S)=0$, and $K_S\sim^{num}
    (\chi(\mathcal O_S)+2g(D/G)-2)F \sim^{num} (2g(D/G)-2)F$, where $F$ is a fibre of $\pi$. So
    $deg(\mathcal N_{\sigma (D/G)/S})=deg (K_S^{-1}\arrowvert _{\sigma (D/G)}\otimes \omega _{\sigma
    (D/G)})=0$, which implies that $\dim H^0(\sigma(D/G),\mathcal N_{\sigma (D/G)/S})\le 1.$ Then
    $\dim  Mor_{\pi}(D/G, S)\le 1$.
    \\But we already have a closed immersion: $$E/G\hookrightarrow Mor_{\pi}(D/G, S),$$ so
    $Mor_{\pi}(D/G, S)$ is locally isomorphic to $E/G$, hence smooth everywhere as $\sigma$ varies,
    which implies $Mor_{\pi}(D/G, S)\simeq E/G.$
    \\We now prove claim 2), our aim is to calculate $K_S\arrowvert _{\sigma(D/G)}. $ \\First, we
    have the following commutative diagram:
    $$\xymatrix {D\ar[r]^f \ar[d]^k &D\times E \ar[d]_h\\D/G \ar[r]^{\sigma}&(D\times E)/G~~,}$$
    where $k$, $h$ are quotient morphisms and $\sigma$ induces a morphism $f_{\sigma}:
    D\longrightarrow E$ satisfying that for all $ x\in D$, $f(x)=(x,f_{\sigma}(x)).$ \\So
    $k^*\circ\sigma^*(K_S)=f^*\circ h^*(K_S)$, we denote this sheaf by  $\mathcal L$, then the sheaf
    we wanted, $K_S\arrowvert _{\sigma(D/G)}=\sigma^*(K_S)=(k_*\mathcal L)^G.$   \\Since $G$ acts
    freely on $D\times E$, $h$ is \'{e}tale so the differential morphism $$dh: h^*\Omega
    _S\longrightarrow \Omega _{D\times E}$$ is an isomorphism, furthermore, $dh$ is a
    $G-$morphism.\\
    Since $\Omega _{D\times E}\simeq p^*\Omega_D\oplus q^*\Omega_E$, so
    $$h^*K_S\simeq^{G}p^*\Omega_D\otimes q^*\Omega_E$$ where ``$\simeq^G$'' means isomorphic  as
    $G-$sheaves. \\Then apply $f^*$ to both sides, we obtain $$\mathcal L=f^*\circ h^*K_S\simeq^G
    f^*(p^*\Omega_D\otimes q^*\Omega_E)=\Omega_D\otimes_{\mathcal O_D} f_{\sigma}^*\Omega _E.$$ Then
    $$\sigma^*K_S\simeq(k_*\mathcal L)^G\simeq k_*(\Omega_D\otimes_{\mathcal O_D} f_{\sigma}^*\Omega
    _E)^G.$$ Since $k: D\longrightarrow D/G$ is \'{e}tale, $\Omega_D\simeq^G k^* \Omega_{D/G}.$\\ By
    the projection formula, $k_*(\Omega_D\otimes_{\mathcal O_D} f_{\sigma}^*\Omega
    _E)\simeq\Omega_{D/G}\otimes_{\mathcal O_{D/G}}k_* f_{\sigma}^*\Omega_E,$ so
    $$\sigma^*K_S\simeq(\Omega_{D/G}\otimes_{\mathcal O_{D/G}}k_* f_{\sigma}^*\Omega_E)^G.$$ Since
    $E$ is an abelian variety, $\Omega_E\simeq ^G\mathcal O_E\otimes _{\mathbb C}H^0(E,\Omega_E)$,
    so $$f^*_{\sigma}\Omega_E\simeq^G f^*_{\sigma}\mathcal O_E\otimes _{\mathbb
    C}H^0(E,\Omega_E)\simeq^G \mathcal O_D\otimes_{\mathbb C}H^0(E,\Omega_E).$$
    Now we have $$\sigma^*K_S\simeq (\Omega_{D/G}\otimes_{\mathcal O_{D/G}}k_* \mathcal
    O_D\otimes_{\mathbb C}H^0(E,\Omega_E))^G.$$ Since rank $\Omega_{D/G}=1$ and $\dim
    H^0(E,\Omega_E)=1,$ we can denote a $G-$invariant element of $\Omega_{D/G}\otimes_{\mathcal
    O_{D/G}}k_* \mathcal O_D\otimes_{\mathbb C}H^0(E,\Omega_E)$ by a pure tensor $a\otimes b\otimes
    c$.\\
    Then for all $ g \in G$, consider the action of $g$ on $a\otimes b\otimes c$:\\
    $g(a\otimes b\otimes c)=a\otimes g b  \otimes g c=a\otimes b\otimes c$. Since $G$ acts on
    $H^0(E,\Omega_E)$ by  a character $\chi\in Hom(G, \mathbb C^*)$, i.e. $g c=\chi(g)c$, then
    $$a\otimes gb\otimes gc=a\otimes gb \otimes \chi(g)c=a\otimes b\otimes c.$$ So
    $gb=\chi(g)^{-1}b$,  let $\mathcal L_{\alpha}=\{a\in k_*\mathcal O_D\arrowvert ga=\alpha(g)a\}$,
    then $$\sigma^*K_S\simeq \Omega_{D/G}\otimes_{\mathcal O_{D/G}}\mathcal L_{\chi^{-1}}.$$ Recall
    a known theorem:(See Chapter 2, Section 7, Proposition 3, ~\cite{Mumford}.)
    \begin{theorem}Assume $G$ acts freely on $X$, and $Y=X/G$. Then for all characters $\alpha:
    G\longrightarrow \mathcal C^*$, $\mathcal L_{\alpha}$ is an invertible sheaf, and the
    multiplication in $\pi_*(\mathcal O_X)$ induces an isomorphism $\mathcal L_{\alpha}\otimes
    \mathcal L_{\beta}\simeq\mathcal L_{\alpha\beta}.$ The assignment $\alpha\mapsto \mathcal
    L_{\alpha}$ defines an isomorphism $\hat{G}\simeq ker(PicY\longrightarrow PicX).$
    \end{theorem}Now apply Theorem 6.2 to $X=D$ and $Y=D/G$, then $\mathcal
    L_{\chi^{-1}}^{-1}=\mathcal L_{\chi}.$ \\By the adjunction formula:$$\mathcal N_{\sigma
    (D/G)/S}\simeq \sigma^*K_S^{-1}\otimes \Omega_{D/G}\simeq  \Omega_{D/G}^{-1}\otimes \mathcal
    L_{\chi}\otimes \Omega_{D/G}\simeq  \mathcal L_{\chi}.$$ By our assumption, some element $g\ne
    id_G$ of $G$ has fixed points on $E$, since $G$ acts on $E$ faithfully, the lineat part of $g$
    is non-trivial, which implies that $\chi\ne 1$. By  Theorem 6.2, $\mathcal L_{\chi}\ne\mathcal
    O_{D/G}\in$ $Pic ~(D/G)$, i.e. $\mathcal N_{\sigma (D/G)/S} \simeq L_{\chi}$ is not trivial. By
    the proof of claim 1), we know that $deg\mathcal N_{\sigma (D/G)/S}=0$. If $H^0(D/G, \mathcal
    N_{\sigma (D/G)/S})\ne0$, then $\mathcal N_{\sigma (D/G)/S}$ is linearly equivalent to an
    effective divisor of degree 0, hence $ \mathcal N_{\sigma (D/G)/S}$ is trivial, we get a
    contradiction. So $H^0(D/G, \mathcal N_{\sigma (D/G)/S})=0$, and this completes our
    proof.\end{proof}
   Now it is easy to see that Theorem 3.10 and Theorem 3.11 are just direct corollaris of Theorem
   6.1.
    Furthermore,  Theorem 6.1 completes our study of $Mor_{\pi}(C,S)$, when $\kappa (S)=1$.

Li Duo,  Academy of Mathematics and Systems Science, CAS, Beijing, China. \\Email address:
liduo211@mails.ucas.ac.cn


\begin{thebibliography}{99}\bibitem{BPV} W. P. Barth, K. Hulek, C. A. M. Peters,  A. Van De Ven. :
\emph{Compact Complex Surfaces. Springer-Verlag, Berlin Heidelberg (2004).}\bibitem{Beauville} A.
~Beauville: \emph{Complex Algebraic Surfaces, Cambridge Univ. Press (1983).}\bibitem{Michel. Brion}
M. Brion: \emph{On Algebraic Semigroups and Monoids. Algebraic monoids, group embeddings, and
algebraic combinatorics, 1-54,
Fields Inst. Commun., 71, Springer, New York, 2014.} \bibitem{Olivier Debarre} Olivier Debarre:
\emph{Higher-Dimensional Algebraic Geometry. Springer-Verlag, New York (2001).} \bibitem{Hartshone}
R. Hartshorne:   \emph{Algebraic Geometry. Springer-Verlag, New York (1977).} \bibitem{Daniel
Huybrechts} Daniel Huybrechts \emph {Lectures on $K3$ surfaces.}
http://www.math.uni-bonn.de/people/huybrech/K3Global.pdf \bibitem{Manin} Manin, Ju. I. :  \emph
{Rational points of algebraic curves over function fields. Izv. Akad. Nauk sssR Ser. Mat. 27.(1963),
1395-1440. }\bibitem{Mumford} David Mumford: \emph{Abelian Varieties. Oxford University Press,
1970}\bibitem{Mumford2} David, Mumford:\emph{Lectures on curves on an algebraic surface, Princeton
University Press, Annals of Math. Studies 59 (1966)} \bibitem{Nagata}M. Nagata: \emph{On
self-intersection number of a section on a ruled surface}, Nagoya Math. J. 37 (1970),
191-302.\bibitem{Putcha} Putcha, M.S.: \emph { Linear Algebraic Monoids. Cambridge University Press,
Cambridge (1988)} \bibitem{Renner} Renner, L.E.: \emph {Linear Algebraic Monoids. Springer-Verlag,
New York (2005)}\bibitem{Wall}  C.T.C. Wall: \emph{Geometric structures on compact complex analytic
surfaces}, Topology 25 (1986), 119-153.
\end{thebibliography}
\end{document}